\documentclass[a4paper,11pt]{amsart}

\usepackage{amsmath, amsthm,amssymb,verbatim,amscd,epsfig,esint,color}
\usepackage[latin1]{inputenc}
\usepackage[T1]{fontenc}
\usepackage{enumerate}

\usepackage[english]{babel}

\theoremstyle{plain}
\newtheorem{thm}{Theorem}[]
\newtheorem{lem}[thm]{Lemma}
\newtheorem{cor}[thm]{Corollary}
\newtheorem{prop}[thm]{Proposition}

\theoremstyle{definition}

\newtheorem{rem}[thm]{Remark}

\numberwithin{equation}{section}

\newcommand{\cC}{\mathcal C}
\newcommand{\C}{\mathbb C}
\newcommand{\D}{\mathbb D}

\newcommand{\R}{\mathbb R}

\newcommand{\cB}{\mathcal B}

\newcommand{\cS}{\mathcal S}
\newcommand{\cH}{\mathcal H}

\newcommand{\cM}{\mathcal M}

\newcommand{\zbar}{\overline{z}}

\renewcommand{\Re}{\operatorname{Re}}

\newcommand{\loc}{\mathrm{loc}}

\let\swap=\phi
\let\phi=\varphi
\let\varphi=\swap
\let\swap=\epsilon
\let\epsilon=\varepsilon
\let\varepsilon=\swap

\let\swap=\leq
\let\leq=\leqslant
\let\leqslant=\swap
\let\swap=\geq
\let\geq=\geqslant
\let\geqslant=\swap

\newcommand{\supp}{\mathrm{supp}}
\newcommand{\dist}{\mathrm{dist}}

\newcommand{\zz}{\mathbb{Z}}

\newcommand{\rr}{\mathbb{R}}
\newcommand{\cc}{\mathbb{C}}
\newcommand{\bz}{\bar{z}}

\newcommand{\hhh}{\mathcal{H}}

\renewcommand\Re{\operatorname{Re}}

\renewcommand\div{\operatorname{div}}

\renewcommand{\bar}[1]{\overline{#1}}

\title[Improved H\"older regularity for strongly elliptic PDEs]{Improved H\"older regularity for \\ strongly elliptic PDEs}
\date{}
\author[K. Astala]{Kari Astala}
\address{K. Astala, Aalto University, Department of Mathematics and Systems Analysis, P.O. Box 11100, FI-00076 Aalto, Finland}
\email{kari.astala@aalto.fi}

\author[A. Clop]{Albert Clop}
\address{A. Clop, Department of Mathematics, Universitat Aut\`onoma de Barcelona, 08193 Bellaterra (Barcelona), Catalonia}
\curraddr{}
\email{albertcp@mat.uab.cat}

\author[D. Faraco]{Daniel Faraco}
\address{D. Faraco, Department of Mathematics, Universidad Aut\'onoma de Madrid, 28049 Ma\-drid, 
Spain; ICMAT CSIC-UAM-UCM-UC3M, 28049 Madrid, Spain}
\curraddr{}
\email{daniel.faraco@uam.es}

\author[J. J\"a\"askel\"ainen]{Jarmo  J\"a\"askel\"ainen}
\address{J. J\"a\"askel\"ainen, University of Jyv\"askyl\"a, Department of Mathematics and Statistics, P.O. Box 35 (MaD), Fi-40014 University of Jyv\"askyl\"a}
\curraddr{}
\email{jarmo.t.jaaskelainen@jyu.fi}

\author[A. Koski]{Aleksis Koski}
\address{A. Koski, University of Jyv\"askyl\"a, Department of Mathematics and Statistics, P.O. Box 35 (MaD), Fi-40014 University of Jyv\"askyl\"a}
\curraddr{}
\email{aleksis.t.koski@jyu.fi}

\thanks{K.A. was supported by the Academy of Finland project SA-13316965  and ERC projects 301179 and 834728. A.C. was partially supported by projects MTM2016-75390 and MTM2016-81703-ERC (Spanish Govt.) and 2017-SGR-395 (Catalan Govt.). D.F. was supported by MTM2014-57769-1-P,  and MTM2017-85934-C3 from the Ministerio de Ciencia e Innovaci\'on (MCINN), by ICMAT Severo Ochoa projects SEV-2011-0087 and  SEV-2015-0554 (MINECO), and by the ERC projects 301179 and 834728. 
J.J. was partially supported by the Academy of Finland (no. 318636). A.K. was supported by the ERC 307023 and the Academy of Finland (no. 315767).}


\begin{document}

\frenchspacing

\begin{abstract} We establish  surprising improved Schauder regularity properties for solutions to  the Leray-Lions divergence type equation  in the plane. The results are achieved by studying  the nonlinear Beltrami equation and making use of special new relations between these two equations. In particular, we show 
 that  solutions to an autonomous Beltrami equation enjoy a quantitative improved degree of H\"older regularity, higher than what is given by the classical exponent $1/K$.

\end{abstract}

\maketitle

\section{Introduction} The theme of this paper is the interaction 
of two elliptic partial differential equations in two dimensions,  the Leray-Lions equation
\begin{equation}\label{diverAuto}
\div A(z, \nabla u) = 0 %\qquad \text{in $\Omega \subset \C$}
\end{equation}
and the nonlinear Beltrami equation 
\begin{equation}\label{bel1}
f_{\zbar} = \cH(z, f_z),
\end{equation}
 as well their inhomogeneous versions, see \eqref{diverEq1} and \eqref{bel2} below. 
  In particular, our aim is to look at the equations and their relations from a new and novel perspective. Secondly, we establish  an unexpected regularity result for the autonomous nonlinear Beltrami equation. As a consequence,  this will lead to improved regularity results for the nonlinear Leray-Lions equations, as well as for  the non-autonomous and inhomogeneous versions of both equations. 
  \smallskip
  
  That the  above two equations are related has been evident for long, see for instance, \cite{Bers-Nirenberg}, \cite{Bojarski}, \cite{Caccioppoli}, \cite{Finn-Serrin},  \cite{Iwaniec-Sbordone},  \cite{Vekua}, and this has been widely used to apply quasiconformal methods to planar elliptic equations,  see for example  the monograph \cite{AIM}. 
   
  However, the point of the  present paper is to give a  new perspective  by finding a sharp ellipticity condition for the structural function  $A(z,\xi)$ under which there is an  equivalence, even up to the exact ellipticity bounds, between the nonlinear equations \eqref{diverAuto} and \eqref{bel1}.  In  the  linear case the precise ellipticity relation was first proved in   \cite{Leonetti-Nesi}. 
  
\smallskip

For the nonlinear Beltrami equation \eqref{bel1}, the structural field $\cH(z,\zeta)$ is assumed measurable in $z \in \Omega$  while 
the ellipticity %of this equation
 is quantified by requiring the uniform Lipschitz bound
\begin{equation}\label{lip31}
| \cH(z, \zeta) - \cH(z, \eta)| \leq k \,|\zeta - \eta|, \qquad \zeta, \eta \in \C, 
\end{equation}
where $0 \leq k < 1$ is a fixed constant. One naturally assumes that $\cH(z,0) \equiv 0$ so that the equation becomes homogeneous. The conditions guarantee that $W^{1,2}_{\loc}$-solutions of \eqref{bel1} are quasiregular. The non-linear Beltrami equation appears also in other scenarios such as study of differential inclusions for gradient maps \cite{tartar}.

For the case of the autonomous Beltrami equation
\begin{equation}\label{autbel1}
f_{\zbar} = \cH(f_z),
\end{equation}
it is known  that the solutions have $K$-quasiregular directional derivatives with $K = \frac{1+k}{1-k}$, \cite{Faraco}, \cite{Faraco-Kristensen}, \cite{Hinkkanen-Martin} \cite{Sverak}, \cite{ACFJK},  and thus by the classical theorem of Morrey \cite{Morrey}, \cite[Section 3.10]{AIM} the solutions lie in $C^{1,1/K}_{\loc}$.   However, surprisingly it turns out their regularity  can be improved even further. Developing this phenomenon leads  to our first key result.

\begin{thm}\label{ImprovedAuto} Under the ellipticity assumption \eqref{lip31},  solutions  $f$ to the autonomous Beltrami equation \eqref{autbel1} belong to $C^{1,\alpha_K}_{\loc}$, where for $K = \frac{1+k}{1-k}$,
\begin{equation}\label{expo}
\alpha_K = \frac{1-k}{1+\frac{k}{2} \max\{1,2-4k\}} = \min\left\{  \frac{4}{3K+1},  \frac{K+1}{3K-1}\right\}
\quad > \frac{1}{K}. 
\end{equation}
\end{thm}

We do not know whether the bound  of Theorem \ref{ImprovedAuto} is sharp. On the other hand, cf. Remark  \ref{HolderRemark},
\[f_0:\cc \to \cc, \quad f_0(z) = z^2|z|^{\frac{3}{2K + 1} - 1}\]
solves an autonomous equation with  ellipticity constant $k=\frac{K-1}{K+1}$, which
 shows that   $C^{1,\alpha}_{\loc}$-regularity can hold only for  $\alpha \leq \frac{3}{2K + 1}$. Indeed, we conjecture this 
 to be the optimal $C^{1,\alpha}_{\loc}$-regularity for general solutions to \eqref{autbel1}. It should also be mentioned that for $C^1$-regular fields $\cH$, the solutions to the autonomous Beltrami equation lie in $C^{1,\beta}$ for all $\beta < 1$, see Theorem 1.3 in the authors' previous paper \cite{ACFJK}.
  \smallskip

 The Leray-Lions  equation   \eqref{diverAuto} has been associated, since the original work  \cite{L-L}, to various notions of monotonicity. A quantified version of this was studied by Kovalev 
\cite{Kovalev}, who considered the $\delta$-monotonous mappings, defined by 
\begin{equation}\label{monot}
 \langle A (z, \xi_1) - A(z,\xi_2), \xi_1 - \xi_2 \rangle \geq \delta |A (z,\xi_1) - A(z,\xi_2)| \,|\xi_1 - \xi_2|, \qquad \xi_i \in \R^2.
\end{equation}
On the other hand, the requirement of uniform ellipticity is typically expressed in terms of the bounds
\begin{equation}\label{ellip}
|\xi|^2 +   |A (z,\xi)|^2 \leq \left( K + \frac{1}{K} \right)    \langle A (z,\xi) , \xi  \rangle, \qquad \xi \in \R^2.
\end{equation}
For $A(z,\xi)$ that is linear in the variable $\xi$, i.e., $A(z,\xi) = A(z)\xi$, with $\det A(z) =1$ the bounds \eqref{monot} and \eqref{ellip} are equivalent for 
$\delta =  \frac{2K}{K^2 +1}$.
In general the notions are different since $\delta$-monotonicity is scale invariant while \eqref{ellip}  is not. 
 
Combining the above notions leads to the following natural requirements for the structural function $A (z,\xi)$,
 \begin{equation}\label{diverEllipCondition}
 \aligned
 &|\xi_1 - \xi_2|^2 + |A(z,\xi_1) - A(z,\xi_2)|^2 \leq \left(K + \frac1K\right) \left\langle \xi_1 - \xi_2 , A(z,\xi_1) - A(z,\xi_2)\right \rangle,
\\
& A(z,0) \equiv 0,
\endaligned
\end{equation}
for $\xi_1, \xi_2 \in \R^2$.
\smallskip

A second key point, shown in Theorem \ref{equivalence}  below, is that under this requirement the Leray-Lions equation \eqref{diverAuto} is equivalent to \eqref{bel1}, including the exact equivalence between the respective ellipticity bounds in \eqref{lip31} and \eqref{diverEllipCondition}.  From this equivalence we, for instance, have  

\begin{cor} \label{LLautonmous} 
Suppose $\Omega \subset \R^2$ is a bounded simply connected domain. If $u \in W^{1, 2}_{\loc}(\Omega, \R)$ is a distributional solution to the Leray-Lions equation $$\div A(\nabla u) = 0,$$ where $A= A(\xi)$ satisfies the strong ellipticity condition \eqref{diverEllipCondition}, then $u$ belongs to the class $C^{1,\alpha_K}_{\loc}(\Omega, \R)$, with $\alpha_K$  defined in \eqref{expo}.
\end{cor}

Accordingly, cf. discussion after Theorem \ref{ImprovedAuto},  we  expect that  the optimal regularity for solutions to the Leray-Lions equations with the structure \eqref{diverEllipCondition} is $C^{1,\alpha}_{\loc}$, where $\alpha = \frac{3}{2K + 1}$.  

\medskip

%\begin{rem}
 When the structure function $A(\xi)$ is  assumed to be $\delta$-monotone, $\delta =  \frac{2K}{K^2 +1}$, there is another route to the improved gradient H\"older regularity for solutions to the autonomous Leray-Lions differential equations: As shown in \cite[Section 16.4]{AIM} the complex gradient of the solution is $K^2$-quasiregular and then \cite{Baernstein-Kovalev} gives an improved $C^{1,\alpha}_{\loc}$-regularity. However, the regularity exponent obtained in this way is smaller than  \eqref{expo}.
 See also \cite{Martin} for yet another recent higher regularity result for the autonomous Beltrami system. 
%\end{rem}

\medskip

It is natural to use these methods for non-autonomous and inhomogeneous equations as well, both for the nonlinear Beltrami equations and the 
Leray-Lions equations. Again the results turn out to be equivalent, including the quantitative bounds. We will concentrate on Schauder type estimates. 

It was shown in the  authors' previous article \cite{ACFJK}  that in Schauder estimates for the  nonlinear Beltrami equation both  smoothness in the variable $z$ as well as the ellipticity constant of the structural field come into play. 
The best smoothness one can obtain is that derivatives of the solutions lie in some H\"older class. Using the above improved $C^{1,\alpha}$-regularity we will then be able to establish the third key result of our paper, which is stated in the form of the following theorem. %the following  general result.
\begin{thm}\label{SchauderforA}
Let $\Omega \subset \R^2$ be simply connected bounded domain. Suppose $u \in W^{1, 2}_{\loc}(\Omega, \R)$ is a distributional solution to  the Leray-Lions equation 
\begin{equation}\label{diverEq1}
\div A(z, \nabla u) = \div g, 
\end{equation}
 where $A$ satisfies the strong ellipticity condition \eqref{diverEllipCondition} and $A$ is $\alpha$-H\"older continuous with respect to the variable $z$, that is, for $z_1, z_2 \in \Omega$ and $\xi \in \R^2$
$$
|A(z_1, \xi) - A(z_2, \xi)| \leq C\,|z_1 - z_2|^\alpha |\xi|.
$$
Assume also that $g$ lies in the class $C^{\alpha}(\Omega,\rr)$. Then $u$ belongs to the class $C^{1,\gamma}_{\loc}$, whenever
$\gamma \leq \alpha$  and $\gamma < \alpha_K  =\min\left\{  \frac{4}{3K+1},  \frac{K+1}{3K-1} \right\}.$
\end{thm}

In fact, in Theorem \ref{SchauderforA} one can replace $\alpha_K$ by $\beta_K$, the largest exponent such that a solution to the autonomous Beltrami equation \eqref{autbel1} always lies in the class $C^{1,\beta_K}_{\loc}$.

In terms of the nonlinear Beltrami equation one has the following  alternative formulation of the theorem. This theorem is a direct strengthening of the main result in \cite{ACFJK}. As a matter of fact one of our motivations to study Schauder estimates in this context was the relation between nonlinear Beltrami equations and nonlinear families of quasiconformal mappings discovered in \cite{ACFJ}, see also \cite{Hinkkanen-Martin}. The results of \cite{ACFJ} are hence also sharpened but we omit the details. % \ref{SchauderforA} reads as follows.

\begin{thm}\label{SchauderforH} Let $\Omega \subset \C$ be a domain, $0 < \alpha < 1$ and $k < 1$. Suppose that the structure field $\cH(z,\zeta)$ satisfies the assumptions
\begin{equation*}
\aligned
&|\cH(z_1, \zeta_1) - \cH(z_2, \zeta_2)| \leq C\,|z_1 - z_2|^{\alpha}\bigl(|\zeta_1| + |\zeta_2|\bigr) + k\,|\zeta_1 - \zeta_2|,\\ 
&\cH(z_1, 0) \equiv 0,
\endaligned
\end{equation*}
for every $z_1, z_2 \in \Omega$ and $\zeta_1, \zeta_2 \in \C$. %Suppose also that $G \in C^{\alpha}(\Omega,\cc)$ is a given function.
Then for all $G \in C^{\alpha}(\Omega,\cc)$, every solution $f : \Omega \to \C$ of the inhomogeneous Beltrami equation \[f_{\zbar} = \cH(z, f_z) + G \qquad \text{a.e. $z\in \Omega$}\] belongs to the regularity class $C^{1,\gamma}_{\loc}$, where $\gamma$ is any number satisfying
$\gamma \leq \alpha$  and $\gamma < \alpha_K.$\end{thm}

In addition to the results already mentioned, we also prove an inhomogeneous version of the Caccioppoli inequality for solutions to nonlinear Beltrami equations for which the field $\cH(z,\zeta)$ satisfies a $\mathrm{VMO}$-regularity condition in the first variable. While we only need this inequality for a step in the proof of Theorem \ref{SchauderforH}, there is a good possibility of such an inequality having future applications in the regularity theory of Beltrami equations as well. This result is stated as Theorem \ref{vmolocalestimate} in the appendix.

\medskip

As a closing remark to the introduction we pose the following open question. As 
our first key result, Theorem \ref{ImprovedAuto}, shows that solutions to the autonomous Beltrami equation \eqref{autbel1} enjoy a degree of H\"older continuity higher than the classical exponent $\frac{1}{K}$, it is natural to ask whether such solutions also lie in a higher Sobolev class. Classical results show that due to the quasiregularity of the directional derivatives, solutions to \eqref{autbel1} lie in the Sobolev class $W^{2,q}_{\loc}(\Omega, \C)$ for all $q < p_K = \frac{2K}{K-1}$. Whether this exponent is the optimal one is a question we leave open, though Theorem \ref{MuConstantLemma} may indicate that  the optimal exponent is larger than $p_K$. %the answer could be positive. 
Nevertheless, there are previous studies that show that solutions to elliptic equations may admit improved H\"older regularity but no improved Sobolev regularity -- see for example \cite{convexintegration}, \cite{Baernstein-Kovalev} and \cite{Koskela}.

\section{Connection to divergence equations} 

Let us begin with the Leray-Lions equation
\begin{equation}\label{diverEq2}
\div A(z, \nabla u) = \div g \qquad \text{in $\Omega \subset \C$}.
\end{equation}
The divergence equation is understood in the distributional sense. Here $u \in W^{1, 2}_{\loc}(\Omega, \R)$ and $A$ is measurable in the $z$-variable. The regularity in the gradient variable defines the ellipticity of the equation.

\begin{thm} \label{equivalence}
Let $\Omega \subset \C$ be simply connected domain. Suppose that $f \in W^{1,2}_{\loc}(\Omega, \C)$ solves the Beltrami equation 
\begin{equation}\label{bel2}
f_{\zbar} = \cH(z, f_z) + G \qquad \text{a.e. $z\in \Omega$}.
\end{equation}
Then $u = \Re f$ solves the Leray-Lions equation \eqref{diverEq2}, where $A$ satisfies the strong ellipticity bound \eqref{diverEllipCondition}.

Conversely, if $u \in W^{1, 2}_{\loc}(\Omega, \R)$ solves \eqref{diverEq2}, where $A$ satisfies \eqref{diverEllipCondition}, then there exists $v\in W^{1, 2}_{\loc}(\Omega, \R)$ such that the function $f = u + iv$ solves a Beltrami equation \eqref{bel2}, where the structure field $\cH(z,\zeta)$ is $k$-Lipschitz in the gradient variable $\zeta$ with
$$ k = \frac{K-1}{K+1}.
$$
\end{thm}

\begin{proof}

Suppose that $f$ solves the Beltrami equation 
$$
f_{\zbar} = \cH(z, f_z) + G \qquad \text{a.e.}.
$$
We will find a divergence type equation for the real part of $f$. Writing $f = u + iv$  we obtain the equation
\begin{equation}\label{realpartHEq} 
u_{\zbar} + iv_{\zbar} = \cH(z,\overline{u_{\zbar} - iv_{\zbar}}) + G.
\end{equation}
Now, we would like to solve the quantity $-iv_{\zbar}$ from the above equation. Note that the mapping $\omega \mapsto  u_{\zbar} - \cH(z,\overline{u_{\zbar} + \omega}) + G$ is a contraction due to the $k$-Lipschitz property of $\cH$. Thus the Banach fixed point theorem gives us a solution of \eqref{realpartHEq} in terms of the variables $z$ and $u_{\zbar}$, i.e., there exists a function $B$ such that
\[-iv_{\zbar} = B(z,u_{\zbar}).\]

To see that $B$ is measurable in the $z$-variable, note that  $B(z, \xi)$ can be obtained by iterating the map $\omega \mapsto  u_{\zbar} - \cH(z,\overline{u_{\zbar} + \omega}) + G$. At each point of the iteration the function is measurable due to Lusin-measurability of $\cH$ and \cite[Theorem 7.7.2]{AIM}, and the limit function is measurable as a pointwise limit of measurable functions.

Note now that the expression $-iv_{{\zbar}} = -\frac12 \left(-v_y + iv_x\right)$ is divergence free. Thus we obtain the divergence type equation
\begin{equation}\label{divergenceEqReal}
\div B(z,u_{\zbar}) = 0,
\end{equation}
which is understood in the distributional sense. To obtain the equation \eqref{diverEq2} with $A$ satisfying the homogeneity condition $A(z,0) \equiv 0$, we may define $g(z) := -B(z,0)$ and $A(z,\xi) := B(z,\xi) + g(z)$.

Let us then show that the function $A(z,\xi)$ satisfies the ellipticity condition \eqref{diverEllipCondition}. It is enough to verify this condition for $B$ since the condition only involves differences with the variable $z$ fixed. Take $\xi_1,\xi_2 \in \C$. Denote $a_j = B(z,\xi_j)$ for $j = 1,2$. Then by \eqref{realpartHEq} we have
\[\xi_j - a_j = \cH(z,\overline{\xi_j + a_j}).\]
Subtracting and taking absolute values gives
\begin{align*}
|\xi_1 - a_1 - \xi_2 + a_2| = |\cH(z,\overline{\xi_1 + a_1}) - \cH(z,\overline{\xi_2 + a_2})| \leq k|\xi_1 + a_1 - \xi_2 - a_2|.
\end{align*}

We next show that this ellipticity condition is equivalent with the strong ellipticity \eqref{diverEllipCondition}. 

\smallskip

\noindent {\bf Claim 1.} The conditions %, that is, the following conditions are equivalent for $\xi_i, a_i \in \C$
\begin{equation}\label{elliptisyysehdot}
%\aligned
%&1. & \qquad
 |\xi_1 - a_1 - \xi_2 + a_2| %&
 \leq k\,|\xi_1 + a_1 - \xi_2 - a_2| %\\
\end{equation}
and 
\begin{equation}\label{elliptisyysehdot2}
%&2. & \qquad 
|\xi_1 - \xi_2|^2 + |a_1 - a_2|^2 %&
\leq \frac{2(1 + k^2)}{1 - k^2} \left\langle \xi_1 - \xi_2 , a_1 - a_2\right \rangle
%\endaligned
\end{equation}
are equivalent for every $\xi_i, a_i \in \C$.

\smallskip

To prove the claim we first take squares on both sides of \eqref {elliptisyysehdot},
\begin{align*}
&|\xi_1 - \xi_2|^2 + |a_1 - a_2|^2 - 2 \Re((\xi_1 - \xi_2)\overline{(a_1 - a_2)})  \\
&\quad \leq k^2\left(|\xi_1 - \xi_2|^2 + |a_1 - a_2|^2 + 2 \Re((\xi_1 - \xi_2)\overline{(a_1 - a_2)})\right)
\end{align*}
and obtain the equivalent inequality
\begin{align*}
|\xi_1 - \xi_2|^2 + |a_1 - a_2|^2 &\leq \frac{2(1 + k^2)}{1 - k^2} \Re((\xi_1 - \xi_2)\overline{(a_1 - a_2)}) \\& =  \frac{2(1 + k^2)}{1 - k^2}\; (\xi_1 - \xi_2) \cdot (a_1 - a_2),
\end{align*}
which equates to the ellipticity condition \eqref {elliptisyysehdot2}. This proves the claim 1.

\smallskip

Hence, by putting $a_j = B(z,\xi_j)$ in \eqref {elliptisyysehdot2}, we see that  $B$ (and hence $A$) satisfies
\begin{equation*}|\xi_1 - \xi_2|^2 + |B(z,\xi_1) - B(z,\xi_2)|^2 \leq \left(K + \frac1K\right) \left\langle \xi_1 - \xi_2 , B(z,\xi_1) - B(z,\xi_2)\right \rangle.\end{equation*}

\bigskip

Conversely, let $u \in W^{1, 2}_{\loc}(\Omega, \R)$ solve \eqref{diverEq2}. As the function $B(z, \nabla u) = A(z, u_{\zbar}) - g(z)$  is divergence free and $\Omega$ is simply connected, by the Poincar\'e lemma there exists $v  \in W^{1, 2}_{\loc}(\Omega, \R)$ such that 
$$
i\,B(z, \nabla u) = v_{\zbar}.
$$

We will show that $f = u + iv$ solves a Beltrami equation. Now,
\begin{equation}\label{defineH}
\begin{cases}
f_{\zbar} = u_{\zbar} + iv_{\zbar} = u_{\zbar} - B(z, u_{\zbar}) \\
\overline{f_{z}} = u_{\zbar} - iv_{\zbar} = u_{\zbar} + B(z, u_{\zbar}).
\end{cases}
\end{equation}
Using the definition of $B$, we find that
\begin{equation}\label{defineH2}f_{\zbar} - g(z)  = u_{\zbar} - A(z, u_{\zbar}) \quad \text{ and } \quad
\overline{f_{z}} + g(z) = u_{\zbar} + A(z, u_{\zbar}).\end{equation}
\smallskip
\noindent {\bf Claim 2.} If a function $\mathcal{A}(z, \xi) : \Omega \times \C \to \C$ satisfies the ellipticity bound \eqref{diverEllipCondition}, then $T_z(\xi) = (\xi + \mathcal{A}(z, \xi)) : \C \to \C$ is invertible. 

\smallskip

The ellipticity bound \eqref{diverEllipCondition} gives that $I + \mathcal{A}(z, \cdot) : \C \to \C$ is strongly monotone, that is, 
\begin{equation}\label{monotoninen}
c\,|\xi_1 - \xi_2|^2 \leq  \left\langle \xi_1 - \xi_2 , \xi_1 + \mathcal{A}(z, \xi_1) - \xi_2  - \mathcal{A}(z, \xi_2)\right \rangle
\end{equation}
and thus coercive and injective. As $I + \mathcal{A}(z, \cdot) : \C \to \C$ is also continuous, bijectivity, and hence the claim 2, follows.

\smallskip

Thus from \eqref{defineH2} we get that 
\begin{equation}\label{AinvEquation}
f_{\zbar} = (I - A(z, \cdot))(I + A(z, \cdot))^{-1}\left(\bar{f_z} + g(z)\right) + g(z).
\end{equation}
We now define a function $\cH^*$ by
\begin{equation}\label{HAsta}
\cH^*(z, \zeta) := (I - A(z, \cdot))(I + A(z, \cdot))^{-1}(\zeta).
\end{equation}
Thus \eqref{AinvEquation} shows that $f$ solves the partial differential equation
\begin{equation}\label{HauxEquation}
f_{\bz} = \cH^*(z,\bar{f_z} + g(z)) + g(z).
\end{equation}
This equation is already of interest by itself, but for the purposes of the result we are proving we need to transform it into a proper Beltrami equation. To this end define
\begin{equation}\label{cHdefined}
\cH(z,\zeta) := \cH^*(z,\bar{\zeta} + g(z)) - \cH^*(z,g(z)) \quad \text{ and } \quad G(z) := \cH^*(z,g(z)) + g(z).
\end{equation}
These definitions with \eqref{HauxEquation} guarantee that $f$ solves the inhomogeneous nonlinear Beltrami equation \eqref{bel2} and we have that $\cH(z,0) \equiv 0$.
We are left to prove the ellipticity and measurability of $\cH$. It is enough to verify these properties for $\cH^*$.

We proved in Claim 1 that  the strong ellipticity \eqref{diverEllipCondition} is equivalent with 
$$
 |\xi_1 - a_1 - \xi_2 + a_2| \leq k\,|\xi_1 + a_1 - \xi_2 - a_2|.
 $$
 Here we let $\xi_i \in \C$ and $a_i = A(z, \xi_i)$. 
 
 Thus, given $\zeta_1, \zeta_2 \in \C$, we choose $\xi_i = (I + A(z, \cdot))^{-1}(\zeta_i)$ and obtain 
 $$
|\cH^*(z, \zeta_1) - \cH^*(z, \zeta_2)| \leq k\,|\zeta_1 - \zeta_2|,
$$ 
that is, the structural field $\cH^*$ (and hence $\cH$) is $k$-Lipschitz in the second variable as wanted.

\smallskip

For the measurability of $\cH^*$  in the $z$-variable, 
we first mollify $A(z,\xi)$ in the variable $z$, 
$$
A^\epsilon(z,\xi)=\int_{\C} \Phi^\epsilon(\eta)A(z-\eta,\xi)\,dm(\eta)
$$
for a positive mollifier $\Phi^\epsilon : \C \to [0, \infty)$. Now, for fixed $\xi \in \C$, the map $z\mapsto A^\epsilon(z,\xi)$ is continuous, and the map converges to $z \mapsto A(z,\xi)$ as $\epsilon\to 0$ for almost every $z$.

The mollified structural function $A^\epsilon$ satisfies the same strong ellipticity bound \eqref{diverEllipCondition} as $B$, since
%$$
%\aligned
%&\left\langle \xi_1 - \xi_2 , A^\epsilon (z,\xi_1) - A^\epsilon (z,\xi_2)\right \rangle \\
%&\quad= 
%\left\langle \xi_1 - \xi_2 , \int_{\C} \Phi^\epsilon(\eta)\Big(A(z-\eta,\xi_1)- A (z - \eta,\xi_2)\Big)\,dm(\eta)\right \rangle \\
%&\quad= 
%\int_{\C} \Phi^\epsilon(\eta)\left\langle \xi_1 - \xi_2 , \Big(A(z-\eta,\xi_1)- A (z - \eta,\xi_2)\Big)\right \rangle \,dm(\eta) \\
%&\quad\geq \frac{1}{\left(K + \frac1K\right)} \int_{\C} \Phi^\epsilon(\eta)\Big( |\xi_1 - \xi_2|^2 + |A(z - \eta,\xi_1) - A(z - \eta,\xi_2)|^2 \Big) \,dm(\eta)\\
%&\quad= \frac{1}{\left(K + \frac1K\right)}\left(|\xi_1 - \xi_2|^2 + \int_{\C} \Phi^\epsilon(\eta)|A(z - \eta,\xi_1) - A(z - \eta,\xi_2)|^2 \,dm(\eta)\ \right) \\
%&\quad \geq  \frac{1}{\left(K + \frac1K\right)}\left(|\xi_1 - \xi_2|^2 + \left(\int_{\C} \Phi^\epsilon(\eta)|A(z - \eta,\xi_1) - A(z - \eta,\xi_2)| \,dm(\eta)\ \right)^2\right) \\
%&\quad \geq  \frac{1}{\left(K + \frac1K\right)}\Big(|\xi_1 - \xi_2|^2 + |A^\epsilon(z,\xi_1) - A^\epsilon(z ,\xi_2)|^2\Big),
%\endaligned
%$$
$$
\aligned
&\left\langle \xi_1 - \xi_2 , A^\epsilon (z,\xi_1) - A^\epsilon (z,\xi_2)\right \rangle \\
&= 
\int_{\C} \Phi^\epsilon(\eta)\left\langle \xi_1 - \xi_2 , A(z-\eta,\xi_1)- A (z - \eta,\xi_2)\right \rangle \,dm(\eta) \\
&\geq \frac{1}{\left(K + \frac1K\right)} \int_{\C} \Phi^\epsilon(\eta)\Big( |\xi_1 - \xi_2|^2 + |A(z - \eta,\xi_1) - A(z - \eta,\xi_2)|^2 \Big) \,dm(\eta)\\
& \geq  \frac{1}{\left(K + \frac1K\right)}\left(|\xi_1 - \xi_2|^2 + \left(\int_{\C} \Phi^\epsilon(\eta)|A(z - \eta,\xi_1) - A(z - \eta,\xi_2)| \,dm(\eta)\ \right)^2\right) \\
& \geq  \frac{1}{\left(K + \frac1K\right)}\Big(|\xi_1 - \xi_2|^2 + |A^\epsilon(z,\xi_1) - A^\epsilon(z ,\xi_2)|^2\Big),
\endaligned
$$
where the second last inequality follows from H\"older's inequality applied to the second term after expanding.

\smallskip

Now, by the proof of Claim 2, $T_{z, \epsilon}  = (I + A^\epsilon(z, \cdot))$ is invertible and satisfies \eqref{monotoninen}. We show that $T^{-1}_{z, \epsilon}$ is continuous in the $z$-variable. In fact given $z_1,z_2 \in \C$ by monotonicity 
\eqref{monotoninen} it holds that 
\begin{equation}\label{jatkuvuuskaanteisoperaattorille}
|T^{-1}_{z_1, \epsilon}(\zeta) - T^{-1}_{z_2, \epsilon}(\zeta)| \leq \frac{1}{c}|\zeta - T_{z_1, \epsilon}(T^{-1}_{z_2, \epsilon}(\zeta))| = \frac{1}{c}|T_{z_1, \epsilon}(\xi_2) - T_{z_2, \epsilon}(\xi_2)|.
\end{equation}
and thus the continuity of $z \mapsto T_{z, \epsilon}(\xi)$  implies that of $z \mapsto T_{z, \epsilon}(\xi)$. Similarly
setting $\xi = T^{-1}_{z}(\zeta)$,
\eqref{monotoninen} implies that
$$
|T^{-1}_{z, \epsilon}(\zeta) - T^{-1}_{z}(\zeta)| \leq \frac{1}{c}|\zeta - T_{z, \epsilon}(T^{-1}_z(\zeta))| = \frac{1}{c}|T_z(\xi) - T_{z, \epsilon}(\xi)| \to 0,
$$
since the convolution $z\mapsto A^\epsilon(z,\xi)$ converges to $z \mapsto A(z,\xi)$  for almost every $z$.
Thus, for every fixed $\zeta \in \C$, $T^{-1}_z(\zeta) = (I+A(z, \cdot))^{-1}(\zeta)$ is measurable in $z$ as it is at almost every point $z$ a limit of continuous functions. We have proved   $(z, \xi) \mapsto (\xi - A(z, \xi))$ is measurable in $z$ and continuous in $\xi$, that is, it is a Carath\'edory function. Hence we conclude that  $\cH^*(z, \zeta) = (I - A(z, \cdot))(I + A(z, \cdot))^{-1}(\zeta)$ is measurable in $z$. 
\end{proof}

\begin{rem}
One can see easily straight from the proof that linearity and autonomity are preserved: If $\cH(z, \zeta)$ is linear in the $\zeta$-variable, then $A(z, \xi)$ is linear with respect to $\xi$ as well (and vice versa). If $\cH$ does not depend on the $z$-variable, neither does $A$ (and vice versa).
\end{rem}

The following proposition will be important for the proof of Theorem \ref{SchauderforA}.

\begin{prop}\label{Holderperiytyy}
If in Theorem \ref{equivalence} the structure function $A$ is also H\"older continuous with respect to $z$, that is,
$$
|A(z_1, \xi) - A(z_2, \xi)| \leq C\,|z_1 - z_2|^\alpha |\xi| \qquad z_1, z_2 \in \Omega, \quad\xi \in \C,
$$
and the inhomogeneous term $g$ is in $C^{\alpha}$, then the auxiliary function $\cH^*$ as defined in \eqref{HAsta} satisfies the Schauder regularity assumptions of Theorem \ref{SchauderforH}.
\end{prop}

\begin{proof}
As we already showed in the proof of Theorem \ref{equivalence}, the function $\cH^*$ is given by \eqref{HAsta}, that is,
$$
\cH^*(z, \zeta) = (I - A(z, \cdot))(I + A(z, \cdot))^{-1}(\zeta)
$$ 
and $\cH^*$ is $k$-Lipschitz in the variable $\zeta$. We are left to show that $\cH^*$ satisfies the condition
\begin{equation}\label{Hholdercond}
|\cH^*(z_1, \zeta_1) - \cH^*(z_2, \zeta_2)| \leq C\,|z_1 - z_2|^{\alpha}\bigl(|\zeta_1| + |\zeta_2|\bigr) + k\,|\zeta_1 - \zeta_2|.
\end{equation}
Using the invertibility of the map $\xi \mapsto \xi + A(z,\xi)$, we define $\xi_j = (I + A(z_j,\cdot))^{-1}(\zeta_j)$ for $j = 1,2$. We now use the inequality \eqref{elliptisyysehdot} which in Claim 1 was found to be equivalent with the strong ellipticity condition on the function $A$ to obtain
\begin{align*}
|(\xi_1 &- A(z_1,\xi_1)) - (\xi_2 - A(z_2,\xi_2)) |\\ &\leq |\xi_1 - \xi_2 + A(z_1,\xi_2) - A(z_1,\xi_1)| + |A(z_2,\xi_2) - A(z_1,\xi_2)|
%\\ &\leq k\,|\xi_1 - \xi_2 + A(z_1,\xi_1) - A(z_1,\xi_2)| + |A(z_2,\xi_2) - A(z_1,\xi_2)|
\\ &\leq k\,|\xi_1 - \xi_2 + A(z_1,\xi_1) - A(z_2,\xi_2)| + (1+k)\,|A(z_2,\xi_2) - A(z_1,\xi_2)|
\\ &\leq k\,|(\xi_1 + A(z_1,\xi_1)) - (\xi_2  + A(z_2,\xi_2))| + C\,|z_1 - z_2|^\alpha |\xi_2|
\end{align*}
Written in terms of $\zeta_1$ and $\zeta_2$, the inequality obtained above reads as
\[|\cH^*(z_1, \zeta_1) - \cH^*(z_2, \zeta_2)| \leq k\,|\zeta_1 - \zeta_2| + C\,|z_1 - z_2|^{\alpha}|\xi_2|.\]
However, the last $|\xi_2|$ should still be estimated above by a constant times $|\zeta_2|$. Hence we seek to obtain the inequality
\[|\xi_2| \leq C\,|\xi_2 + A(z_2,\xi_2)|\]
for some constant $C$ uniformly in $z_2$. This comes from applying the strong monotonicity \eqref{monotoninen} of $A$ with $\xi_1 = 0$ and $z = z_2$ to find that
\[c|\xi_2|^2 \leq \left\langle \xi_2 , \xi_2  + A(z, \xi_2)\right \rangle\leq |\xi_2|\,|\xi_2  + A(z, \xi_2)|.\]
Dividing by $c|\xi_2|$ gives the desired inequality. Hence the auxiliary field $\cH^*$ satisfies \eqref{Hholdercond} as desired.
\end{proof}

\section{Directional derivatives}

The following proposition,  \cite[Proposition 2.1]{ACFJK}, tells us that directional derivatives play a key role in the study of autonomous Beltrami equations.
\begin{prop}\label{directionalderivativesKqr} Let $f$ be a solution to an autonomous Beltrami equation \eqref{autbel1}. Then the directional derivatives of $f$ are $K$-quasiregular, which can be summarized in the condition:
\begin{equation}\label{dirqc}|f_{z\bz} + \theta f_{\bz\bz}| \leq k |f_{zz} + \theta f_{z\bz}|\end{equation}
at almost every point $z$ and for every unit vector $\theta$. 
\end{prop}

\begin{proof}

In \cite[Proposition 2.1]{ACFJK} it was proved that the directional derivatives are quasiregular and thus, for fixed $\theta$, \eqref{dirqc} holds almost everywhere.
 The exceptional set might a priori depend on $\theta$. However,  by quasiregularity the directional derivatives are  differentiable almost everywhere and continuous. Therefore $\partial_x,\partial_y$ are continuous and simultaneously differentiable up to a null set, and  so are $\partial_z f, \partial_{\overline{z}} f$ and in
fact all directional derivatives $\partial_{\theta} f$.

Due to continuity, the difference quotient $f_{h,\theta}=\frac{f(z+h\theta)-f(z)}{h}$ satisfies the distortion inequality for
every $z \in \C$:
\[ |(\partial_{\overline z} f)_{h,\theta}| \le  k|(\partial_{ z} f)_{h,\theta}| . \]
In particular, at the points of differentiability of $\partial_z  f$ and $\partial_{\overline{z}} f$ we have $|  \partial_{\theta}\partial_{\overline z} f |\le k | \partial_{\theta} \partial_{ z}  f |$, which is \eqref{dirqc}. \end{proof}

%\bigskip

%Now given $z \in \C$, the triangle inequality yields that,
%\begin{align*} &|\partial_{\overline z} \partial_{\theta} f | \le k  |\partial_{ z} \partial_{\theta} f |+ |\partial_{\overline z}(\partial_{\theta} f-f_{h,\theta})|+|\partial_z (\partial_{\theta} f-f_{h,\theta})| \\ &=k  |\partial_{ z} \partial_{\theta} f |+ |\partial_{\theta}(\partial_{\overline z} f)-(\partial_{\overline z}f)_{h,\theta})|+|\partial_{\theta} (\partial_z f)-(\partial_z f)_{h,\theta})|
%\end{align*}
%where the last change of order of differentiation is legal at the points of differentiability of $\partial_x f,\partial_y f$. Finally
%at those points we can let $h$ tend to $0$ to get, 

%\[ |\partial_{\overline z} \partial_{\theta} f |\le k |\partial_{ z} \partial_{\theta} f | \]
%which is \eqref{dirqc}.

The above proposition reveals that solutions to autonomous equations enjoy one more degree of regularity than general quasiregular maps. In particular any such solution must be in the classes $C^{1,1/K}_{\loc}$ and $W^{2,p}_{\loc}$ for $p < \frac{2K}{K-1}$ due to classical regularity theory of quasiregular maps. But the proposition also indicates that to study solutions to autonomous equations, we should first understand those maps which have $K$-quasiregular directional derivatives. The following lemma indicates that these two classes of maps are not that different.
\begin{lem}\label{DirectionalLemma} Suppose $f \in W^{2,2}_{\loc}(\Omega)$ has $K$-quasiregular directional derivatives. Then the following hold
\begin{enumerate}
\item{The map $f_z : \Omega \to \cc$ is $K$-quasiregular.}
\item{If $f$ is $K$-quasiregular, then at those points where $f_z$ is a local homeomorphism $f$ is locally the solution of an autonomous equation.}
\end{enumerate}
\end{lem}
Note that in the above lemma, the set of points where $f_z$ is not a local homeomorphism is quite small. More precisely, since $f_z$ is quasiregular the Sto\"ilow factorization reveals that this set is the zero set of a holomorphic function.

\begin{proof}[Proof of Lemma \ref{DirectionalLemma}]
To prove the quasiregularity of $f_z$, we start by manipulating \eqref{dirqc}. Choosing $\theta$ suitably and applying the triangle inequality gives us the following inequalities almost everywhere:
\begin{eqnarray} 
|f_{z\bz}| + |f_{\bz\bz}|&\leq& k |f_{zz}| + k |f_{z\bz}|, \label{dist34}\\
|f_{z\bz}| - |f_{\bz\bz}|&\leq& k||f_{zz}| - |f_{z\bz}||. \label{dist56}
\end{eqnarray}
Suppose first that $|f_{z\bz}| > |f_{zz}|$. Then adding the above inequalities together would give that $2|f_{z\bz}| < 2k |f_{z\bz}|$, a contradiction. Thus $|f_{z\bz}|\leq |f_{zz}|$. Adding the two inequalities now gives
\[2|f_{z\bz}| \leq 2k |f_{z z}|,\]
which is the distortion inequality for $f_z$, and thus $f_z$ is $K$-quasiregular.

Suppose now that $f_z$ is injective in an open set $V \subset \Omega$. Let $z_0 \in V$. Then around the point $\omega_0 = f_z(z_0)$ we may define the field $\hhh$ by $\hhh := f_{\bz} \circ f_z^{-1}$, giving us the autonomous equation $f_{\bz} = \hhh(f_z)$. Now at $z_0$, we define two linear transformations by
\[A_{z_0}(z) = f_{zz}(z_0) z + f_{z\bz}(z_0) \bar{z} \]
and
\[B_{z_0}(z) = f_{z\bz}(z_0) z + f_{\bz\bz}(z_0) \bar{z}.\]
In fact, $A_{z_0}$ and $B_{z_0}$ are the linear differentials of $f_z$ and $f_{\bz}$ at $z_0$ respectively. The $K$-quasiregularity of the directional derivatives of $f$ implies that $|B_{z_0}(z)| \leq k |A_{z_0}(z)|$ for every complex number $z$. Since $f_z$ is quasiconformal at $z_0$, the map $A_{z_0}$ is invertible. Hence the linear map $B_{z_0} \circ A_{z_0}^{-1}$ is well-defined and its operator norm is bounded by the number $k$. We may now calculate that
\[|D\hhh(\omega_0)| = |D f_{\bz}(f_z^{-1}(\omega_0))|\,|D f_{z}^{-1}(\omega_0)| = |B_{z_0}| |A_{z_0}^{-1}| \leq k.\]
The above inequality shows that the map $\hhh$ is $k$-Lipschitz at any point $\omega_0 \in U$, where $U = f_z(V)$. Since the equation $f_{\bz} = \hhh(f_z)$ holds in the set $V$, the $K$-quasiregularity of $f$ implies that the inequality $|\hhh(\zeta)|\leq k|\zeta|$ also holds for all $\zeta \in U$. Thus $\hhh$ is $k$-Lipschitz in the (possibly disconnected) set $U \cup \{0\}$ where we define $\hhh(0) = 0$ if $0 \notin U$. We may now extend $\hhh$ as a $k$-Lipschitz mapping to the whole plane by the Kirszbraun extension theorem, and this guarantees that it satisfies the required normalization $\hhh(0) = 0$. This proves our claim.
\end{proof}
Another equivalent condition for a map to be a solution of an autonomous Beltrami equation is given by the quasiregularity of the increments, as shown in the following result.
\begin{prop} A function $f$ is the solution of an autonomous Beltrami equation if and only if the increments $f(z+v)-f(z)$ are $K$-quasiregular functions of $z$.
\end{prop}
\begin{proof}
If $f$ solves an autonomous Beltrami equation, then the inequality
\[|f_{\bz}(z+v) - f_{\bz}(z)| = \left| \hhh(f_z(z+v)) - \hhh(f_z(z)) \right| \leq k |f_{z}(z+v) - f_{z}(z)|\]
proves the claim in the other direction. To establish the converse, note that since the increment $f(z+v)-f(z)$ is quasiregular, $f$ belongs to $C_{\loc}^{1, 1/K}$ by the proof of \cite[Proposition 2.1]{ACFJK}. Thus we have that at every point $z$
\[f_{z}(z+v) - f_{z}(z) = 0 \qquad \Rightarrow \qquad f_{\bz}(z+v) - f_{\bz}(z) = 0.\]
This implies that if $f_z(z) = f_z(\omega)$ for some $z,\omega\in\cc$, then also $f_{\bz}(z) = f_{\bz}(\omega)$. As a consequence $f$ uniquely defines a field $\hhh$ such that $\cH(f_z(z))=f_{\bar{z}}(z)$. This field $\cH$ is $k$-Lipschitz in the image set of the function $f_z$ by construction and the $K$-quasiregularity of the increments. We may extend $\cH$ to the whole complex plane as in the proof of Lemma \ref{DirectionalLemma}, part (2), giving us the autonomous equation solved by $f$.
\end{proof}

\begin{lem}\label{MuNuLemma} Let $f$ have $K$-quasiregular directional derivatives. Denote $\mu = f_{z\bz}/f_{zz}$ and $\nu = f_{\bz\bz}/f_{zz}$. Then we have that almost everywhere
\begin{enumerate}
\item{$|\nu| \leq k + (k-1)|\mu|$, in particular $|\nu| \leq k$.}
\item{$|\nu-\mu^2| \leq \frac1k (k^2 - |\mu|^2)$.}
\end{enumerate}
\end{lem}

\begin{proof}%[Proof of Lemma \ref{MuNuLemma}]
We know from Lemma \ref{DirectionalLemma} that $|f_{zz}| \neq 0$ outside a discrete set.
Dividing \eqref{dist34} by this gives $|\mu| + |\nu| \leq k(1+ |\mu|)$ which proves the first claim. 

For  
($2$), dividing \eqref{dirqc} with $|f_{zz}|$ shows that the inequality
$ |\mu + \theta \nu| \leq k | 1 + \theta \mu|$ holds for every parameter $ \theta \in S(1) = \partial \D(0, 1)$. Thus for almost every fixed $z$,
$$ \Theta: \theta \mapsto \frac{\mu(z) + \theta \nu(z)}{1 + \theta \mu(z)}
$$
is an analytic function from the unit disc onto $\D(0,k)$.  In particular, with the Schwarz lemma one has $\, k |\Theta'(0)| \leq k^2- |\Theta(0)|^2$, which gives the second claim.
\end{proof}

We now come to our main regularity result. We will make use of the following  auxiliary real valued functions

\[I_f := -i \bigl[k  \, {\overline{\, f_z}\,} \, \partial_\phi f_z + 
\overline{ \,f_{\overline{z}} \,} \, \partial_\phi  f_{\overline{z}}  \bigr]\]
and 
\[j_f := k |f_{zz}|^2 + (1-k)|f_{z\bz}|^2 - |f_{\bz\bz}|^2 = k J_{f_z} + J_{f_{\bz}}.\]
Above $\partial_\phi g$ denotes the angular derivative of a function $g$. Recall that in the complex notation it has  the representation
\begin{equation} \label{angular3}
\partial_\phi g = i(zg_z - \bz g_{\bz}), \qquad z = |z|e^{i\phi}.
\end{equation}

By a direct application of the Cauchy-Schwarz and Poincar\'e inequalities to each of the terms, the following inequality will holds with constant $1$ in from of $\int_{S(r)} |\partial_\phi  f_{\overline{z}} | ^2$
 but due to the relations between $f_z$ and $f_{\overline z}$ there is an improvement.
\begin{lem}\label{Poincare} Suppose  $f \in W^{2,2}_{\loc}(\Omega)  $ and let $z_0 \in \Omega$.
Then for almost all small radii, the  circles $S(r) = S(z_0,r)$ satisfy
\begin{equation}
\int_{S(r)} I_f    \le k  \int_{S(r)} |\partial_\phi  f_{{z}} | ^2
+\max\left\{\frac{1}{2},1-2k\right\} 
\int_{S(r)} |\partial_\phi  f_{\overline{z}} | ^2
\end{equation}
\end{lem}

\begin{proof}
Let us express $f_z$ and $f_{\bz}$ as a Fourier series  so that for  fixed radius $r$,
\begin{equation*} f_z(re^{i\phi}) = \sum_{n\in\zz} A_n (r) e^{in\phi}, \quad f_{\bz}(re^{i\phi}) =\sum_{n\in\zz} B_n (r) e^{i n \phi}.\end{equation*}

Notice that, since $\int_{S(r)} \partial_\phi f d\phi =0$, it follows that

\begin{equation}
\label{constant}
A_{-1}=B_{1}, \end{equation}
which is what gives us further improvement. 
Taking the angular derivative of the Fourier series and using theorems of Plancherel and Parceval we get 

\begin{align*}
&\frac{1}{2\pi r} \int_{S(r)} I_f=\sum_{n\in\zz} n(k|A_n|^2 + |B_n|^2), \quad {\rm while}
\\ & \frac{1}{2\pi r} \int_{S(r)} k \, | \,\partial_\phi  f_{{z}}\, |^2 +  |\,\partial_\phi  f_{\overline{z}}\,|^2 
=\sum_{n\in\zz} n^2 (k|A_n|^2 +  |B_n|^2).
\end{align*}

Now, \eqref{constant} implies that,
\begin{align*}
&\sum_{n\in\zz} n (k|A_n|^2 + |B_n|^2) = %\\ &
 k\sum_{-1 \neq n \, \in \, \zz} n |A_n|^2+ \, (1-k)|B_1|^2 \, +\sum_{1 \neq n \, \in \, \zz} n |B_n|^2.
\end{align*}
Therefore, noticing that $n\le \frac{1}{2}n^2$ when  $n\neq\pm 1$,  we can estimate  %$(1-k)=\frac{1}{K}(1+k)$,  and \frac{1}{K}

\begin{align*}
&\sum_{n\in\zz} n (k|A_n|^2 + |B_n|^2)\le \\ 
& \quad k\sum_{-1 \neq n \, \in \, \zz} n^2 |A_n|^2+ k |A_{-1}|^2+ (1-2k) |B_1|^2+\frac{1}{2}\sum_{1 \neq n \, \in \, \zz} n^2 |B_n|^2,
\end{align*}
which tells that 
\[\int_{S(r)} I_f \le k\int_{S(r)} | \,\partial_\phi  f_{{z}} \, |^2+ \max\left\{\frac{1}{2}, 1-2k\right\}\int_{S(r)} |\partial_\phi f_{\overline{z}}|^2, \]
proving the claim.
\end{proof}
Notice that $f(z)=|z|^2z$ yields equality in the above lemma. However $\partial_z f$ is not quasiregular.  On the other hand the map
$g(z)=\frac12(z+k\bz)^2$ has quasiregular directional derivatives and shows that the bound $1-2k$ is sharp as $k$ tends to zero. 

\begin{lem} \label{pointwise3}
Suppose $f \in W^{2,2}_{\loc}(\Omega)  $ has $K$-quasiregular directional derivatives. Then for almost every $z = |z|e^{i\phi}$, 
\begin{equation} \label{pointwise}
|\,\partial_\phi f_z(z)\,|^2 \le \frac{|z|^2}{k(1-k)} \, j_f(z)
\end{equation}
\end{lem}

\begin{proof} In view of  \eqref{angular3}, we are to show that 
\[ \frac{1}{|z|^2}|zf_{zz} - \bz f_{z\bz}|^2 \leq \frac{1}{k(1-k)}(k|f_{zz}|^2 + (1-k)|f_{z\bz}|^2 - |f_{\bz\bz}|^2).\]
 In terms of the coefficients $\mu$ and $\nu$ from Lemma \ref{MuNuLemma}, this requires  
\begin{equation*}
k(1-k) (1 + |\mu|)^2 \leq k + (1-k)|\mu|^2 - |\nu|^2.
%
%\frac{k + (1-k)|\mu|^2 - |\nu|^2}{|1 + \theta \mu|^2} \geq k(1-k) \ \ \text{ where } \theta= - \frac{\bz}{z} \text{ is a  unit vector}.
\end{equation*}
But that follows  immediately  by squaring the  first  bound   $|\nu| \leq k + (k-1)|\mu|$ of  Lemma \ref{MuNuLemma}.
\end{proof}
%,  it holds that 

%\begin{align*} & \alpha k + (1-k)|\mu|^2 - |\nu|^2 \alpha \ge k + (1-k)|\mu|^2 - (k + (k-1)|\mu|)^2\\
% &=k(1-k)(1+ |\mu|)^2)\ge k(1-k)|1 + \theta \mu|^2,\end{align*}
%where the last inequality follows from triangle inequality. The proof is concluded. 

\begin{thm}\label{ImprovedRegularity} Suppose $f$ is a mapping whose directional derivatives are $K$-quasiregular. Then $f$ lies in the class $C^{1,\alpha_K}$, where $\alpha_K > 1/K$. In particular one may take
\[\alpha_K = \frac{1-k}{1+\frac{k}{2} \max\{1,2-4k\}} = \min\left\{  \frac{4}{3K+1},  \frac{K+1}{3K-1}\right\}.\]
\end{thm}

\begin{rem} \label{HolderRemark} While the above theorem gives an improvement on the classical H\"older regularity exponent $1/K$, we do not know what is the sharpest possible exponent for these mappings. Our current best example is the function $f(z) = z^2 |z|^{\frac{3}{2K + 1} - 1}$, which solves a $k$-Lipschitz autonomous Beltrami equation in the whole plane. Indeed, to construct the example consider first the function  $f(z) = z^2 |z|^{2\alpha}$, where $-\frac{1}{2} < \alpha < 1$. For this map
$$ f_z(z) = (2+\alpha) z |z|^{2\alpha} \quad {\rm with } \quad f_{\bar z}(z) = \alpha \frac{z^3}{|z|^2} |z|^{2\alpha}. % \quad \Rightarrow \quad f_
$$
Therefore our function satisfies $ f_{\bar z} = \hhh(f_z)$, where 
$$ \hhh(w) = \frac{\alpha}{2 + \alpha} \frac{w^3}{|w|^2} \quad \Rightarrow \quad k := Lip(\hhh) = \frac{3|\alpha|}{2+\alpha}.
$$
In particular, we have $2\alpha + 1 = \frac{3}{2K + 1}$ for $K= \frac{1+k}{1-k}$. Hence for this function $f \in C^{1,\gamma}$ with $\gamma = \frac{3}{2K + 1}$. This smoothness is still strictly better than the class $C^{1,\alpha_K}$ obtained in Theorem \ref{ImprovedRegularity}.
\end{rem}

\begin{proof}[Proof of Theorem \ref{ImprovedRegularity}]
We assume that $k \neq 0$, as in the case $k = 0$ one trivially has smooth regularity.

Our proof is a modification Morrey's classical proof of the $C^{1/K}_{\loc}$-regularity of quasiregular maps. This  utilizes the Morrey-Campanato characterization of H\"older continuity, essentially requiring us to prove that for a small disc $\D_R = \D(z_0,R)$ and all radii $r \in (0,R)$ we have
\begin{equation}\label{MorreyCamp}
\int_{\D_r} |D^2 f|^2 dm(z) \leq C \left(\frac{r}{R}\right)^{2\alpha_K} \int_{\D_R} |D^2 f|^2 dm(z),
\end{equation}
where $\alpha_K$ is our desired H\"older exponent. Lemmas \ref{DirectionalLemma} and \ref{MuNuLemma} imply that $|f_{z\bz}|$ and $|f_{\bz\bz}|$ are both dominated by $|f_{zz}|$. Thus we are  free to replace the quantity $|D^2 f|^2$ in \eqref{MorreyCamp} by the expression
\[j_f = k |f_{zz}|^2 + (1-k)|f_{z\bz}|^2 - |f_{\bz\bz}|^2 = k J_{f_z} + J_{f_{\bz}}.\]

We now denote $J(r) = \int_{\D_r} j_f(z) dm(z)$. To obtain \eqref{MorreyCamp} it will be enough to prove the following inequality,
\[J'(r) \geq \frac{2\alpha_K}{r} J(r) \ \ \text{ for } r \in (0,R). \]

To show the inequality use first  Green's formula to obtain $$J(r)= \frac{1}{2 r}\int_{S(r)} I_f \,|dz|.$$ Second, by  quasiregularity of directional derivatives,  
for almost every $r$ we have $f \in W^{2,2}(S(r))$ while \eqref{dirqc}  and \eqref{angular3} show that for almost every $z \in S(r)$, 
 \begin{equation}\label{angularqr}
|\,\partial_\phi f_{\overline{z}}\,| \leq k | \,\partial_\phi  f_{{z}} \, |. 
 \end{equation}
We can combine this with
our improved Poincar\'e inequality Lemma~\ref{Poincare} to see that 
\begin{align*} 2 r J(r) &\leq  k\int_{S(r)} | \,\partial_\phi  f_{{z}} \, |^2+ \max\left\{ \frac{1}{2}, 1-2k \right\} \int_{S(r)} |\,\partial_\phi f_{\overline{z}}\,|^2
\\ &\leq \left( k+ k^2 \max\left\{ \frac{1}{2}, 1-2k \right\}  \right) \int_{S(r)} | \,\partial_\phi  f_{{z}}\, |^2 
\end{align*}
where the last line follows from \eqref{angularqr}. This allows us to complete the proof with the  pointwise estimate of  $| \,\partial_\phi  f_{{z}} \,|^2$ from Lemma \ref{pointwise3}. In conclusion, we obtain 
\[ J(r) \leq \frac{r}{2\alpha_K}  \int_{S(r)} j_f \,|dz| =\frac{r}{2\alpha_K} J'(r) \qedhere\]
\end{proof}

\begin{proof}[Proof of Theorem \ref{ImprovedAuto}]
By Proposition \ref{directionalderivativesKqr} the directional derivatives of a solution $f$ to an autonomous Beltrami equation \eqref{autbel1} are $K$-quasiregular. Thus the claim follows from Theorem \ref{ImprovedRegularity}.
\end{proof}
%
%\begin{proof}[Proof of Theorem \ref{SchauderforA}]
%By Theorem \ref{equivalence} and Proposition \ref{Holderperiytyy} there exists $v$ such that $f = u + iv$ solves a nonlinear Beltrami equation such that the structural field $\cH$ satisfies the Schauder-regularity assumptions of Theorem \ref{SchauderforH}. Now, the claim follows Theorem \ref{SchauderforH} and Theorem \ref{ImprovedRegularityMain}.
%\end{proof}

Our next theorem will reveal another reason why it is nontrivial to find the optimal regularity for mappings with $K$-quasiregular directional derivatives. Recall that for general $K$-quasiregular mappings, the number $K$ determines their degree of regularity. In other words, the smaller the Beltrami coefficient $\mu$ is the more regular is the map. Hence to obtain an extremal quasiregular map we should certainly require that $||\mu||_\infty = k$, and for the usual radial stretching example we indeed have the equality $|\mu| = k$ everywhere. Surprisingly, however, this equality cannot hold for the Beltrami coefficient of $f_z$ for a map with $K$-quasiregular directional derivatives -- unless the map is smooth.

\begin{thm} \label{MuConstantLemma} Suppose that $f$ has $K$-quasiregular directional derivatives, and 
 that $\mu = f_{z\bz}/f_{zz}$, the Beltrami coefficient of $f_z$,  satisfies $|\mu|=k$ in a disc $D_0 \subset \C$. 
 
 Then in that disc, $\mu = ke^{i\phi_0}$ is constant and  $f$ is real analytic. In fact, %, up to an affine transformation, 
 %equal to the smooth map 
 $f(z) =  \Phi(z + ke^{i\phi_0} \bz)$, where $\Phi$ is holomorphic.
\end{thm}

\begin{proof} By the second part of Lemma \ref{MuNuLemma},   if $|\mu| = k$ in  $D_0$, then $\nu = \mu^2$ in this disc. Hence we obtain the following two equations:
\begin{equation} \label{twoBeltramis}
f_{z\bz} = \mu f_{zz} \ \ \text{ and } \ \ f_{\bz\bz} = \mu^2 f_{zz} = \mu f_{z\bz} \qquad \text{a.e. \, in \,} D_0.
\end{equation}
These imply that the functions $f_z$ and $f_{\bz}$ are both solutions to the same Beltrami equation in the disc $D_0$. Hence for a homeomorphic solution $h$ of this equation, $h_{\bar z} = \mu h_z$, we may find analytic functions $A$ and $B$ in $h(D_0)$ so that
\begin{equation} \label{analytic}
f_z = A \circ h \ \ \text{ and } \ \ f_{\bz} = B \circ h.
\end{equation}
This implies that $(A \circ h)_{\bz} = (B \circ h)_z$, which simplifies to
\begin{equation}\label{muequation}(A' \circ h) \mu = B' \circ h \ \ \text{ a.e. \, in \,} D_0.\end{equation}
Since $|\mu| = k$, the holomorphic function $B'/A'$ has constant modulus in the open set $h(D_0)$. Thus
%But since this set is open, the function 
$B'/A'$ must be constant. This shows that $\mu$ is a constant, and one can then solve the equations \eqref{twoBeltramis} to find that $f = \Phi(z + ke^{i\phi_0} \bz)$ for some holomorphic $\Phi$ as claimed.
\end{proof}

\begin{rem} From the above arguments one can get little more, namely for $f$ to be real analytic  it is enough that $\nu=\mu^2$ a.e. in $D_0$.  Even more, if the disc $D_0$ is such that quasiregular map $f_z$ is  invertible in $D_0$, then the second part of Lemma  
\ref{DirectionalLemma} gives a  $k$-Lipschitz structure function $\hhh$ such that
\begin{equation} \label{analytic3} 
f_{\bar z} = \hhh( f_z) \qquad {\rm in \;} D_0.
\end{equation}
But then \eqref{analytic} implies that $B = \hhh \circ A$, so that $\hhh$ must be complex analytic. This implies that
$f$ is real analytic.  

Indeed, a quick way to see this is by first  derivating \eqref{analytic3} and using \eqref{twoBeltramis} to show that $\mu = \hhh'(f_z)$. Derivating  in turn the last identity gives
% implies that $\mu=\Psi \circ h$ where $\Psi=B'/A'$ is holomorphic. Thus $\mu$ solves 
\[ \mu_{\bz}=\mu \,\mu_z \]
so that $\mu$ is its own Beltrami coefficient. In particular, $\mu = \phi \circ h$ for some analytic function $\phi$.
Finally, cf. \cite[p.34]{AIM}, the inverse $g = h^{-1}$ satisfies  the linear equation $g_{\bar z} = - (\mu \circ h^{-1}) {\overline {g_z}}$, 
that is $g_{\bar z} = - \phi {\overline{ g_z}}$. In particular, $g$ is real analytic, and as a homeomorphism it has non-vanishing Jacobian. It follows that $h$ and hence by \eqref{analytic} also $f$ must be real analytic.  

We also note that the autonomous Beltrami equation \eqref{analytic3} with a complex analytic $\hhh$ is a central tool in the recent study 
 \cite{ADPZ} of the geometry scaling limits of random tilings. 
%This implies that $\mu$ is quasiregular, therefore $C^{\frac{1}{K}}$, and by classical Schauder estimates \cite{AIM} a bootstrapping argument   yields the $C^\infty$-smoothness.  
\end{rem}

In \cite{Uniqueness} it was shown that a global quasiconformal solution to an autonomous equation \eqref{autbel1} must be an affine map. We are able to make the same conclusion while only assuming the quasiregularity of the directional derivatives, giving a completely new proof for the case of the autonomous equation as well.
\begin{thm}\label{GlobalAffine}
Let $f : \cc \to \cc$ be a global quasiconformal map with $K$-quasiregular directional derivatives. Then $f$ is an affine map.
\end{thm}
The proof of this theorem utilizes the following lemma, which we state by itself as it is also a general statement about maps with $K$-quasiregular directional derivatives.
\begin{lem}\label{posJacLemma} Let $f : \Omega \to \cc$ be a map with $K$-quasiregular directional derivatives. Then the Jacobian of $f$ is strictly positive at the points $z_0 \in \Omega$ where $f$ is a local homeomorphism.
\end{lem}
\begin{proof}
This proof is essentially done in \cite[Theorem 3.3]{ACFJK}. While the cited result is stated only for solutions to autonomous Beltrami equations, on closer inspection the proof only utilizes the fact that the directional derivatives are $K$-quasiregular. Hence the same proof also applies to the statement of Lemma \ref{posJacLemma}.
\end{proof}
We are now ready to prove Theorem \ref{GlobalAffine}.
\begin{proof}[Proof of Theorem \ref{GlobalAffine}]
First of all, we recall from Lemma \ref{DirectionalLemma} that we know that $f_z:\cc \to \cc$ is a quasiregular map. Hence $f_z = A \circ h$ for some entire function $A$ and a global quasiconformal map $h$ with $h(0) = 0$ by Sto\"ilow's factorization. Since $f$ is a homeomorphism, Lemma \ref{posJacLemma} implies that the Jacobian $|f_z|^2 - |f_{\bz}|^2$ is positive. In particular $A \neq 0$ everywhere. If $A$ is a constant function then $f$ is affine and we would be done, so let us assume the contrary.

From the above we conclude that $A$ is a nonvanishing, nonconstant entire function. Thus in particular it has an essential singularity at infinity, and the holomorphic function $\tilde{A}(z) = A(1/z)$ has an essential singularity at $z = 0$. However, since global quasiconformal maps extend quasiconformally to the whole Riemann sphere, we also know that the function $\tilde{f}(z) = 1/f(1/z)$ is quasiconformal at $z=0$. If we also denote $\tilde{h}(z) = 1/h(1/z)$, we may compute the Cauchy-Riemann derivative of $\tilde{f}$ as follows.
\[\frac{\partial \tilde{f}}{\partial z} (z) = \frac{f_z(1/z)}{z^2 f(1/z)^2} = \frac{f_z(1/z)}{z^2 f(1/z)^2} =  \frac{\tilde{A}\circ \tilde{h}(z)}{z^2 f(1/z)^2}.\]

Since $\tilde{f}$ is quasiconformal at $z=0$, its derivatives are locally in $L^2$. In a small disc $D_0$ we thus have that
\[\int_{D_0} \frac{|\tilde{A}\circ \tilde{h}(z)|^2}{|z|^4 |f(1/z)|^4} dm(z) < \infty.\]
Quasiconformality of $f$ gives the estimate $|f(z)| \leq C |z|^K$ for large $z$. Applying this estimate in $D_0$ gives that for the exponent $\alpha = 4(K-1)$ we have
\[\int_{D_0} |\tilde{A}\circ \tilde{h}(z)|^2 |z|^{\alpha} dm(z) < \infty.\]
We can now get rid of the factor $|z|^{\alpha}$ by using H\"older's inequality, utilizing the fact that $|z|^{-\beta}$ is integrable in $D_0$ for $\beta < 1$. This implies that
\[\int_{D_0} |\tilde{A}\circ \tilde{h}(z)|^\gamma dm(z) < \infty\]
for some positive number $\gamma$. We may also make a change of variables to find out that
\[\int_{D_1} |\tilde{A}(\omega)|^\gamma J_{\tilde{h}^{-1}}(\omega)dm(\omega) < \infty\]
for some small disc $D_1$ centered at zero. We can use H\"older's again to get rid of the extra factor $J_{\tilde{h}^{-1}}$, utilizing the fact that by the quasiconformality of $\tilde{h}^{-1}$ the expression $J_{\tilde{h}^{-1}}^{-\beta}$ is locally integrable for small positive $\beta$. This again leads to
\[\int_{D_1} |\tilde{A}(\omega)|^{\tilde{\gamma}}dm(\omega) < \infty\]
for some positive number $\tilde{\gamma}$. 

This will be a contradiction since $A$ is a holomorphic function with an essential singularity at zero, and thus it cannot be integrable to any positive power. This fact is true in general but in our case we can also simply use the fact that $A$ doesn't vanish to see that the function $B = A^{\tilde{\gamma}/2}$ is a well-defined holomorphic function that is locally in $L^2$ about the origin. Now by looking at the Laurent series expansion of $B$ one quickly sees that $B$ must have a removable singularity at $z = 0$. This gives a contradiction with the fact that $A$ has an essential singularity, and proves our claim.
\end{proof}

\section{General Schauder estimates}

This section is devoted to proving Theorem \ref{SchauderforH}, that is to establish  Schauder estimates of $W^{1,2}$-solutions to the inhomogeneous nonlinear Beltrami equation 
\begin{equation}\label{hqrjac}
f_{\zbar} = \cH(z, f_{z}) + G(z)  \qquad \text{a.e. \, in} \;\;  \Omega.
\end{equation}
Recall that  the strong ellipticity of the equation is encoded in the fact that the function $(z, \zeta) \mapsto \cH(z, \zeta)$ is $k$-Lipchitz in the variable $\zeta$,
where $k<1$.

Let us recall our regularity assumptions on the structure field $\cH$. Throughout this section we will assume H\"older continuity of $\cH$ in the variable $z$ and $k$-Lipschitz dependence on the variable $\zeta$. More precisely, given an open bounded set $\Omega \subset \C$, we assume that 
\begin{equation}\label{Hcondition}
\aligned
&|\cH(z_1, \zeta_1) - \cH(z_2, \zeta_2)|\leq \mathbf{H}_{\alpha}(\Omega)|z_1 - z_2|^{\alpha}(|\zeta_1| + |\zeta_2|) + k\,|\zeta_1 - \zeta_2|,\\
&\cH(z_1, 0) \equiv 0,
\endaligned
\end{equation}
for all $z_1, z_2 \in \Omega$, $\zeta_1, \zeta_2 \in \C$, where $\alpha \in (0,1)$ and  $k = \frac{K - 1}{K + 1} < 1$ are fixed. 

When $\cH$ is linear in the gradient variable (i.e., the Beltrami equation is linear), \eqref{Hcondition} implies 
that the derivatives of the solutions to Beltrami equation  are $\alpha$-H\"older continuous (\cite{LU}, see also \cite[Chapter 15]{AIM}). Our goal is to show that similar regularity results hold in the general nonlinear case.

\begin{thm}\label{schauder} 
Assume that the structure field $\cH$ satisfies the assumptions \eqref{Hcondition}. Suppose also that $G \in C^{\alpha}(\Omega,\cc)$ is a given function.
Then any solution $f : \Omega \to \C$ of the equation \eqref{hqrjac} lies in the regularity class $C^{1,\gamma}_{\loc}$, where $\gamma$ is any positive number satisfying
\begin{itemize}
\item{$\gamma \leq \alpha$ and }
\item{$\gamma < \beta_K$, where $\beta_K$ is the largest exponent such that a solution to the autonomous Beltrami equation \eqref{autbel1} always lies in the class $C^{1,\beta_K}_{\loc}$.}
\end{itemize}

Moreover, we have a norm bound, when  $\D(\omega, 2R) \Subset \Omega$,
\begin{equation}\label{thmnorm}
{\|D_zf\|}_{C^\gamma(\D(\omega, R))}
\leq c(K, \alpha, \gamma, \omega, R, \mathbf{H}_\alpha(\Omega))\left(\|D_z f\|_{L^2(\D(\omega, 2R))} +  \|G \|_{C^{\alpha}(\D(\omega, 2R))}\right)\!.\end{equation}

\end{thm}
With data $G \equiv 0$ Theorem~\ref{schauder} is proven in \cite{ACFJK}, and our proof follows the same line of arguments.

%The restriction $\gamma < \alpha_{K}$ is not needed if in addition to \eqref{Hcondition}, the structure function $\cH$ is assumed to be $C^1$ in the gradient variable as well. This will follow from the fact that for a $C^1$-regular autonomous equation, the solutions will be shown to be in $C^{1,\beta}_{\loc}(\Omega, \C)$ for every $0 < \beta < 1$. 
 We will first recall the regularity results of the autonomous case (shown in the previous sections)  and then,  in the spirit of classical Schauder estimates \cite{s1}, \cite{s2}, tackle the general case by perturbation.

\subsection{The autonomous equation and integral estimates}\label{autosec}

%In the case of nonlinear Beltrami equations with constant coefficients  $\cH$ depends only on the gradient variable, and the requirement \eqref{Hcondition} reduces to $\cH(0)=0$ with $|\cH(\xi_1)-\cH(\xi_2)| \leq k|\xi_1-\xi_2|$.

\begin{prop}\label{ccc}
Let $F  \in W^{1,2}_{\loc}(\Omega, \C)$ be a solution to the  autonomous inhomogeneous nonlinear Beltrami equation 
\begin{equation}\label{auto}
F_{\zbar} = \cH(F_{z})  + c_0 \qquad \text{for a.e. $z\in \Omega$},
\end{equation}
where $c_0$ is a complex constant. 
Then the directional derivatives of $F$ are $K$-quasiregular, $K = \frac{1 + k}{1 - k}$.
\end{prop}

 Since $c_0$ does not depend on $z$, 
the difference quotients 
$$
F_h(z) := \frac{F(z + h\theta) - F(z)}{h}, \qquad |\theta| = 1, \quad h > 0. 
$$
are $K$-quasiregular as in Proposition \ref{directionalderivativesKqr}, \cite[Proposition 2.1]{ACFJK}. Therefore,  the directional derivatives inherit the properties of $K$-quasiregular maps. In particular, Theorem \ref{ImprovedRegularity} implies that the derivative $D_zF$ of a solution to the autonomous equation \eqref{auto} is locally $\alpha_K$-H\"older continuous. For perturbation arguments it is particularly useful to formulate this in a Morrey-Campanato form.

\begin{cor}\label{holder1K}
 If $F$ is as in Proposition~\ref{ccc}, the derivative $D_z F$ is  locally $\alpha_K$-H\"older continuous. Moreover,
\begin{enumerate}
\item for every $\D(z_0, \rho) \subset \D(z_0, R) \subset \Omega$,
$$
\| D_z F \|_{L^2(\D(z_0, \rho))} \leq c(K)\,\frac{\rho}{R}\;\| D_z F \|_{L^2(\D(z_0, R))}.
$$
\item For every $\D(z_0, \rho) \subset \D(z_0, R) \subset \Omega$,
 $$
 \|D_z F-(D_z F)_\rho\|_{L^2(\D(z_0,\rho))} \leq  c(K)\left(\frac\rho{R}\right)^{1 +  \alpha_{K}}\,\|D_z F-(D_z F)_{R}\|_{L^2(\D(z_0,R))}
$$
 where $(D_z F)_r = \fint_{\D(z_0, r)} D_z F$.
\end{enumerate}
\end{cor}

\subsection{The Riemann-Hilbert problem}

The solution of the Riemann Hilbert problem is well-known; the proof is based on the local versions of the classical Cauchy transform and the Beurling transform.

\begin{prop}\label{splitting}
Let  $f$ be a solution to the inhomogeneous nonlinear Beltrami equation \eqref{hqrjac}, and suppose $\D(z_0, R) \Subset \Omega$. Then there exists a unique solution $F \in W^{1, 2}(\D(z_0, R), \C)$ to the following local Riemann-Hilbert problem for the inhomogeneous autonomous equation
\begin{equation}\label{Split}
\begin{cases}
F_{\zbar} = \cH(z_0,F_{z})  + (G)_R & \text{a.e. $z \in \D(z_0, R)$}, \\
\Re(f - F) = 0 & \text{on $\partial \D(z_0, R)$}
\end{cases}
\end{equation}
where $(G)_R = \fint_{\D(z_0, R)} G$. 
Furthermore, $\| F_{\zbar}  - f_{\zbar} \|_{L^2(\D_R)} = \| F_{z} - f_z\|_{L^2(\D_R)}$ and we have a norm bound
\begin{equation}\label{L2bound}
\|D_z F \|_{L^2(\D_R)} \leq 2 K \|D_z f \|_{L^2(\D_R)} + \|G - (G)_R\|_{L^2(\D_R)}.
\end{equation}
\end{prop}

\begin{proof}
The local Cauchy transform in $\D_R:= \D(z_0, R)$ is in our case
$$
(\cC_{\D_R} \psi)(z) = \frac1{\pi} \int_{\Omega} \left(\frac{\psi(\zeta)}{z - \zeta} -\frac{(z - z_0)\,\overline{\psi(\zeta)}}{R^2 - (z - z_0)\, \overline{(\zeta - z_0)}} \right)dm(\zeta),
$$
for $\psi \in L^2(\D_R, \C)$,
and the local Beurling transform $\cS_{\D_R} \psi = \partial_z\, \cC_{\D_R} \psi$, that is,
$$
(\cS_{\D_R} \psi)(z) = -\frac{1}{\pi} \int_{\Omega} \left(\frac{\psi(\zeta)}{(z - \zeta)^2} +\frac{R^2\,\overline{\psi(\zeta)}}{(R^2 - (z - z_0)\, \overline{(\zeta - z_0)})^2} \right)dm(\zeta).
$$

The isometry of $\cS_{\D_R}$ implies that the Beltrami operator
$$
(\cB\psi)(z) = \cH(z_0,(\cS_{\D_R}\psi)(z) + f_{z}) - \cH(z,f_{z}) - G(z) - (G)_R
$$
is a contraction on $L^2(\D_R, \C)$. The rest of the proof follows as in \cite[Proposition 2.4]{ACFJK}.
\end{proof}

\subsection{Schauder estimates by freezing the coefficients, Theorem~\ref{SchauderforH}}

We will use the Morrey-Campanato integral characterization of  H\"older continuous functions
\cite[Chapter III, Theorem 1.2, p. 70, and Theorem 1.3, p. 72]{Gia}. Namely,  the integral estimate
\begin{equation}\label{morreycamp}
\|g - g_{\rho} \|_{L^2(\D(z_0, \rho))} \leq M\,\rho^{1 + \gamma}
\end{equation}
for $z_0 \in \Omega$ and every $\rho \leq \min\{R_0, \mathrm{dist}(z_0, \partial\Omega)\}$ (for some $R_0$) gives the local $\gamma$-H\"older continuity of $g$ in $\Omega$. Moreover, for $\tilde{\Omega}\Subset\Omega$, \eqref{morreycamp} implies the H\"older seminorm bound
\begin{equation}\label{semi}
[g]_{C^{\gamma}(\tilde{\Omega})} \leq c(\gamma, \tilde{\Omega})\,M
\end{equation}
and the $L^\infty$-bound
\begin{equation}\label{sup}
\| g \|_{L^\infty(\tilde{\Omega})} \leq c(\gamma, \tilde{\Omega})\left(M\,\mathrm{diam}(\Omega)^\gamma + \|g\|_{L^2(\Omega)}\right)\!,
\end{equation}
see the proofs of Proposition 1.2 and Theorem 1.2 in pages 68--72 of \cite[Chapter III]{Gia}.

Next, we apply the ideas of freezing the coefficients to get few basic estimates for solutions to \eqref{hqrjac}. We start with the following

\begin{lem} \label{basicII}
Suppose $\cH$ satisfies the conditions \eqref{Hcondition} and $G \in C^\alpha(\Omega, \C)$. Let $f \in W^{1, 2}_{\loc}(\Omega, \C)$ be a solution to 
\begin{equation*}
f_{\zbar} = \cH(z, f_{z}) + G(z) \qquad \text{a.e. \, in} \;\;  \Omega.
\end{equation*}
If $\D(z_0, R) \Subset \Omega$, then for each $0 < \rho \leq R$ we have
$$
\aligned \| D_z f - (D_zf )_\rho \|_{L^2(\D_\rho)}  &\leq  c(K) \left(\frac{\rho}{R}\right)^{1+\alpha_{K}} \|D_z f - (D_zf )_R\|_{L^2(\D_R)}\\
&\; + c(K) \,  \mathbf{H}_\alpha(\Omega)\, R^\alpha\, \| f_z \|_{L^2(\D_R)} + c(K)\,[G]_{C^\alpha(\D_R)}R^{1 + \alpha} \endaligned
$$
where $\D_r = \D(z_0, r)$.
\end{lem}

\begin{proof} The required estimate to prove is the same as  in Corollary \ref{holder1K}, claim $(2)$, up to the correction term   
$c(K)\,  ( \mathbf{H}_\alpha(\Omega) R^\alpha \| f_z \|_{L^2(\D_R)} + [G]_{C^\alpha(\D_R)}R^{1 + \alpha})$. This will  arise from a comparison of $f$ and the solution $F$ to an inhomogeneous autonomous equation, the local Riemann-Hilbert problem
$$
\begin{cases}
F_{\zbar} = \cH(z_0, F_{z}) + (G)_R  & \text{a.e. $z \in \D_R$}, \\
\Re(f - F) = 0 & \text{on $\partial \D_R$}.
\end{cases}
$$
The existence of $F$  follows by Proposition~\ref{splitting}. Furthermore, by \eqref{Hcondition}, \\
%The listed  properties of the solution follow from the mapping properties of these 
%operators in combination with $\alpha$-H\"older continuity 
%of $\cH$ and Cacciopoli's inequality for $f$. Indeed,
$$
\aligned
&\|(f-F)_{\zbar} \|_{L^2(\D_R)} \\
&\qquad \leq \|\cH(z, f_z)-\cH(z_0, f_z)\|_{L^2(\D_R)} + \|\cH(z_0, f_z)-\cH(z_0,F_z)\|_{L^2(\D_R)} \\
&\qquad \quad + \|G - (G)_R\|_{L^2(\D_R)} \\
&\qquad\leq\mathbf{H}_\alpha(\Omega)\,R^\alpha\,\|f_z\|_{L^2(\D_R)} +k\, \|(f- F)_z\|_{L^2(\D_R)}  + \|G - (G)_R\|_{L^2(\D_R)} .
\endaligned
$$\\
%Above we use the $k$-Lipschitz property and the $\alpha$-H\"older continuity of the structure function $\cH$. 
Since the Beurling transform  $\cS_{\D_R}$ of the disk $\D_R$ is an isometry $L^2(\D_R) \to L^2(\D_R)$ and $\|G - (G)_R\|_{L^2(\D_R)} \leq 2^\alpha \sqrt{\pi}\; [G]_{C^\alpha(\D_R)}R^{1 + \alpha}$, we end up with
\begin{equation} \label{basicI}
\| D_z f - D_z F  \|_{L^2(\D_R)}  \leq \frac{1 + k}{1 - k}\,  \mathbf{H}_\alpha(\Omega)\,R^\alpha\,\| f_z\|_{L^2(\D_R)} + \frac{2^{1 + \alpha} \sqrt{\pi}}{1 - k}\, [G]_{C^\alpha(\D_R)}R^{1 + \alpha}.
\end{equation}
\vspace{-.3cm}

On the other hand, Corollary  \ref{holder1K} $(2)$ gives 
$$
\aligned
&\| D_z f - (D_zf )_\rho \|_{L^2(\D_\rho)}  \leq \| D_z F - (D_zF )_\rho \|_{L^2(\D_\rho)} + 2 \| D_z f - D_z F  \|_{L^2(\D_\rho)}\\
&\quad\leq c(K)\left(\frac\rho{R}\right)^{1 + \alpha_{K}}\,\|D_z F-(D_z F)_{R}\|_{L^2(\D_R)} + 2 \| D_z f - D_z F  \|_{L^2(\D_R)}\\
&\quad\leq c(K)\left(\frac\rho{R}\right)^{1 + \alpha_{K}}\,\|D_z f-(D_z f)_{R}\|_{L^2(\D_R)}\\
&\qquad \quad + (2\,c(K) + 2) \| D_z f - D_z F  \|_{L^2(\D_R)}, 
\endaligned
$$
$\rho \leq R.$ Combining this  with \eqref{basicI} gives the claim.
\end{proof}

If we use  claim $(1)$ of Corollary  \ref{holder1K}, instead of  claim $(2)$,  the same argument as above leads to 
\begin{lem} \label{basicIV}
Suppose $\cH$ satisfies the conditions \eqref{Hcondition} and $G \in C^\alpha(\Omega, \C)$. If $f \in W^{1, 2}_{\loc}(\Omega, \C)$ and   $\D(z_0, R)$ are as in Lemma
\ref{basicII}, then for each $0 < \rho \leq R$,
$$ \aligned \| D_z f \|_{L^2(\D_\rho)} &\leq c(K)\,  \frac{\rho}{R}\,  \| D_z f \|_{L^2(\D_R)}\\
&\qquad + c(K) \,  \mathbf{H}_\alpha(\Omega)\, R^\alpha\, \|  f_z \|_{L^2(\D_R)} + c(K)\,[G]_{C^\alpha(\D_R)}R^{1 + \alpha}.\endaligned
$$
\end{lem}
\smallskip

If the data $G$ is not zero, $W^{1,2}_{\loc}$-solutions  to \eqref{hqrjac} are not a priori $K$-quasiregular and we don't have the Caccioppoli estimates immediately at our use. The H\"older continuity of the coefficients let us, though, derive Caccioppoli type estimates. These are convenient to present in the following form.

\begin{lem} \label{cacciopp2}
 Suppose $\cH$, $G \in C^\alpha(\Omega, \C)$ and $f \in W^{1, 2}_{\loc}(\Omega, \C)$   are as in Lemma \ref{basicII}. Let $\D(z_0, R) \subset \Omega'' \Subset \Omega' \Subset \Omega$. If $f\in C^\beta (\Omega', \C)$ for some $0 < \beta \leq 1$, then
$$ \| D_z f \|_{L^2(\D(z_0, R))} \leq c(K, \Omega', \Omega'', \mathbf{H}_\alpha(\Omega)) \left( [f]_{C^\beta(\Omega')} \, R^\beta + \|G\|_{C^{\alpha}(\Omega')}\,R\right)\!.$$
\end{lem}

\begin{proof} By Theorem \ref{vmolocalestimate} in the appendix, we have the following Caccioppoli type estimate
$$
\aligned
\| D_z f \|_{L^2(\D(z_0, R))} & \leq c_1\left(\frac{1}{\mathrm{dist}(\partial\Omega'', \partial\Omega')R}\,\|f - (f)_R \|_{L^2(\Omega')} + \|G\|_{L^2(\Omega')}\right)\\
& \leq c_2\left( [f]_{C^\beta(\Omega')} \, R^\beta +  \|G\|_{L^\infty(\Omega')} \, R\right).
\endaligned
$$
where the constants $c_1$ and $c_2$ depend on $K$, $\Omega'$, $\Omega''$ and $\mathbf{H}_\alpha(\Omega)$.
\end{proof}

Lastly, let us  recall
\begin{lem}[Lemma 2.1, p. 86, in  {\cite[Chapter III]{Gia}}]\label{growth}
Let $\Psi$ be non-negative, non-decreasing function such that
$$
\Psi(\rho) \leq a\left[\left(\frac{\rho}{R}\right)^{\lambda} + \sigma\right]\Psi(R) + bR^{\gamma}
$$
for every $0 < \rho \leq R \leq R_0$, where $a$ is non-negative constant and $0 < \gamma < \lambda$. Then there exists $\sigma_0 = \sigma_0(a, \lambda, \gamma)$ such that, if $\sigma < \sigma_0$,
$$
\Psi(\rho) \leq c(a, \lambda, \gamma)\left[\left(\frac{\rho}{R}\right)^{\gamma}\Psi(R) + b\rho^{\gamma}\right]
$$
for all $0 < \rho \leq R \leq R_0$.
\end{lem}

\medskip

With these tools and estimates at our disposal we are ready for   the Schauder estimates.

\begin{proof}[Proof of Theorems~\ref{SchauderforH} and \ref{schauder}]  Denote $\D_r = \D(z_0, r)$. As we are dealing with local estimates we can assume that $\Omega$ is bounded.
\medskip

\noindent{\em Step 1. H\"older continuity of $f$.}\quad We will show that $f$ is actually locally $\beta$-H\"older continuous for every $0 < \beta < 1$. 

Namely, according to Lemma \ref{basicIV} we have
\begin{equation*} \label{guiqintaI} 
\| D_z f \|_{L^2(\D_\rho)} \leq c_0(K) \left( \frac{\rho}{R} \, +  \,  \mathbf{H}_\alpha(\Omega) R^\alpha \right)  \| D_z f \|_{L^2(\D_R)} + c_1(K)\,[G]_{C^\alpha(\D_R)}R^{1 + \alpha},
\end{equation*}
whenever $0 < \rho \leq R$ and $\D_R = \D(z_0, R) \subset \Omega$.  Applying Lemma \ref{growth} to $\Psi(\rho) = \| D_z f \|_{L^2(\D_\rho)}$, with $b= c_1(K)\,[G]_{C^\alpha(\D_R)}, \lambda = 1$ and $\sigma= \mathbf{H}_\alpha(\Omega) R^\alpha$, we see that 
$$\| D_z f \|_{L^2(\D_\rho)} \leq c_2(K, \epsilon)  \left(\left( \frac{\rho}{R} \right)^{1-\epsilon}  \| D_z f \|_{L^2(\D_{R})} + [G]_{C^\alpha(\D_R)}\,\rho^{1 - \epsilon} \right), 
$$
where $0 < \rho \leq R \leq \min \{R_0, \dist(z_0, \partial \Omega)\}$. 
Here $R_0 \leq 1$ is small enough; how small $R_0$ needs to be taken depends  on $c_0(K),  \mathbf{H}_\alpha(\Omega)$ and $\epsilon >0$ but not on $z_0$. Thus  the same upper bound  $R_0$ works throughout the bounded domain $\Omega$. 

Combining with the Poincar\'e inequality gives
$$\aligned \|f-f_\rho\|_{L^2(\D_\rho)}  &\leq  \rho\, \| D_z f \|_{L^2(\D_\rho)}\\
&\leq c_2(K, \epsilon)\,\rho^{2-\epsilon}\left(R^{\epsilon-1}\, \| D_z f \|_{L^2(\D_R)} + [G]_{C^\alpha(\D_R)} \right),
\endaligned$$
for $0 < \rho \leq R \leq \min \{R_0, \dist(z_0, \partial \Omega)\}$. 

Let $\D(\omega, 4R) \subset \Omega$. Now, for $\D(z_0, \rho) \subset \D(\omega, 2R)$,
$$\aligned &\|f-f_\rho\|_{L^2(\D_\rho)} \\ &\quad  \leq c_2(K, \epsilon)\, \rho^{2-\epsilon} \left(\min\{R_0, R \}^{\epsilon - 1} \| D_z f \|_{L^2(\D(\omega, 3R))}+ [G]_{C^\alpha(\D(\omega, 3R))} \right).
\endaligned
$$
In view of \eqref{morreycamp} we see that $f \in C^\beta_{\loc}(\D(\omega, 2R),\C)$ for every $0 < \beta = 1 - \epsilon < 1$.  The estimate \eqref{semi} gives a bound for the local H\"older norm,
\begin{equation}\label{fholdernorm}
 [f]_{C^\beta(\D(\omega, R))} \leq c_3(K, \beta, R, \mathbf{H}_\alpha(\Omega)) \left( \| D_z f \|_{L^2(\D(\omega, 3R))}  +  [G]_{C^\alpha(\D(\omega, 3R))} \right).
 \end{equation}

\medskip

\noindent{\em Step 2: Self-improving Morrey-Campanato estimate.}\quad 
Claim: Assume that $1<\alpha+\beta<1+\alpha_{K}$. Then  $D_z f\in C^{\alpha+\beta-1}_{\loc}(\Omega, \C)$. 

\smallskip

Let $\Omega'' \Subset \Omega' \Subset \Omega$. We first show the claim for $\beta < 1$, and start with estimates from Lemma \ref{basicII}, 
$$
\aligned &\| D_z f - (D_zf )_\rho \|_{L^2(\D_\rho)}  \leq  c_0(K) \left(\frac{\rho}{R}\right)^{1+\alpha_{K}} \|D_z f - (D_zf )_R\|_{L^2(\D_R)}\\
&\qquad + c_0(K) \,  \mathbf{H}_\alpha(\Omega)\, R^\alpha\, \|  f_z \|_{L^2(\D_R)} + c_0(K)\,[G]_{C^\alpha(\Omega')}R^{1 + \alpha} \endaligned
$$
when $ \D(z_0, R) \subset \Omega''$. Here, by the Caccioppoli estimate of Lemma \ref{cacciopp2} 
\begin{equation}
\label{boundedness}
 \| \partial_z f \|_{L^2(\D_R)} \leq c_1(K, \Omega', \Omega'', \mathbf{H}_\alpha(\Omega)) \left( [f]_{C^\beta(\Omega')} \, R^\beta + \|G\|_{C^{\alpha}(\Omega')}\,R\right)\!, 
\end{equation} 
which by Step 1 is finite for every $\beta < 1$.

We will now apply Lemma \ref{growth} to the non-decreasing function $\Psi(\rho) = \|D_z f-(D_z f)_\rho\|_{L^2(\D_\rho)} = \inf_{a\in \C} \| D_z f - a \|_{L^2(\D_\rho)}$ and the parameters  $\lambda = 1+ \alpha_{K}$, $\sigma = 0$ and $b = c_2([f]_{C^\beta(\Omega')} + \|G\|_{C^{\alpha}(\Omega')})$, where $c_2(K, \Omega', \Omega'', \mathbf{H}_\alpha(\Omega))$. We obtain that
\begin{equation}\label{gammanorm} 
\aligned
\| D_z f - (D_zf )_\rho \|_{L^2(\D_\rho)}  &\leq  c_3 \left(\frac{\rho}{R}\right)^{\alpha+\beta}  \|D_z f - (D_zf )_R\|_{L^2(\D_R)} \\
&\quad + c_3 \, \rho^{\alpha+\beta} \left(  \,[f]_{C^\beta(\Omega')} + \|G\|_{C^{\alpha}(\Omega')} \right)
\endaligned
\end{equation}
whenever $\rho\leq R$. 

In terms of the Morrey-Campanato estimate \eqref{morreycamp} in the set $\Omega''$, we see that 
$D_z f\in C^{\alpha+\beta-1}_{\loc}(\Omega'', \C)$, which is enough for our claim if $\alpha \geq \alpha_{K}$. The norm estimate \eqref{thmnorm} follows  
from combining \eqref{semi} with  \eqref{fholdernorm} and \eqref{gammanorm}.

In case $\alpha < \alpha_{K}$ we need to continue to show that $f \in C^{1, \alpha}_{\loc}(\Omega, \C)$. But  what we have  above proves that 
 $D_z f$ is locally bounded. Thus the bound  in \eqref{boundedness} remains finite for $\beta = 1$, and we can repeat the proof of \eqref{gammanorm} with $\beta = 1$. Accordingly, \eqref{morreycamp} and \eqref{semi}  give $f \in C^{1, \alpha}_{\loc}(\Omega, \C)$, with norm bound

 $$
 \aligned
{[D_zf]}_{C^\alpha(\D(\omega, R))}
&\leq c(K, \alpha, \omega, R, \mathbf{H}_\alpha(\Omega))\biggl[\|D_z f\|_{L^2(\D(\omega, 2R))}\\
&\qquad    + \|D_z f\|_{L^\infty(\D(\omega, 2R))} + \|G\|_{C^{\alpha}(\Omega')}\biggr].\endaligned
$$

To estimate the $L^\infty$-norm in $\D(\omega, 2R)$, we note that for $\D(z_0, \rho) \subset \D(\omega, \frac{5R}{2})$ \eqref{gammanorm} holds with $\Omega' = \D(\omega, 3R)$ and thus once more by Morrey-Campanato norm estimate \eqref{morreycamp} (with \eqref{sup})
$$
\aligned
\|D_z f\|_{L^\infty(\D(\omega, 2R))} &\leq c(K, \alpha, \omega, R, \mathbf{H}_\alpha(\Omega))\biggl[\|D_z f \|_{L^2(\D(\omega, 3R))}\\
&\qquad +  [f]_{C^{\beta'}(\D(\omega, 3R))} + \|G\|_{C^{\alpha}(\D(\omega, 3R))}\biggr],
\endaligned
$$
where $\beta'< 1$. It remains to combine with  \eqref{fholdernorm} to obtain
$$
{\|D_zf\|}_{C^\gamma(\D(\omega, R))}
\leq c(K, \alpha, \gamma, \omega, R, \mathbf{H}_\alpha(\Omega))\,\|D_z f\|_{L^2(\D(\omega, 9R))} + \|G\|_{C^{\alpha}(\D(\omega, 9R))},
$$
and we have  the norm bound  \eqref{thmnorm} by rescaling.
\end{proof}

\section{Schauder regularity for the Leray-Lions equation}

In this section we prove Theorem \ref{SchauderforA}.
\begin{proof}[Proof of Theorem \ref{SchauderforA}] Let $u$ be a solution of the Leray-Lions equation \eqref{diverEq1}. We apply Theorem \ref{equivalence} and Proposition \ref{Holderperiytyy} to find that $f = u + iv$ is a solution of the equation \eqref{HauxEquation}, which we recall to be
\[
f_{\bz} = \cH^*(z,\bar{f_z} + g(z)) + g(z).
\]
Let now $\D(z_0,2r) \subset \Omega$. We aim to show that $f$ lies in the regularity class $C^{1,\gamma_{\alpha,K}}(\D(z_0,r))$ where $\gamma_{\alpha,K}$ is the exponent from Theorem \ref{SchauderforH}. Let hence $\psi$ be a smooth function equal to $1$ in $\D(z_0,r)$ and $0$ outside $\D(z_0,2r)$. With $\mathcal{C}$ denoting the Cauchy transform in the plane, we define
\[f^*(z) = f(z) + \bar{\mathcal{C}g(z)}.\]
By definition, $f^*_{z} = f_z + \bar{g}$. Hence we see that $f^*$ is a solution of the nonlinear Beltrami equation
\[f^*_{\bz} = \cH^*(z,\bar{f^*_z}) + g + \bar{\mathcal{S} g},\]
where $\mathcal{S}$ denotes the Beurling transform. Now by Proposition \ref{Holderperiytyy} the field $\cH^*$ satisfies the assumptions of Theorem \ref{SchauderforH}, and the inhomogeneous term $g + \bar{\mathcal{S} g}$ lies in $C^{\alpha}(\D(z_0,r))$ since the Beurling transform maps $C^{\alpha}(\cc)$ to itself. Thus Theorem \ref{SchauderforH} implies that $f^*$ lies in the regularity class $C^{1,\gamma}$, which shows that the same holds for $f$ and $u = \Re(f)$ as well. This concludes the proof.
\end{proof}

\appendix

\section{The Caccioppli estimate}

In this appendix, 
our assumption is that the structure function $\cH$ is \emph{locally uniformly in $\mathrm{VMO}$}, that is, if $\mathcal{K}\subset\Omega$ is compact then
\begin{equation}\label{locunifvmo}
\lim_{R\to 0}\sup_{a\in \mathcal{K}}\sup_{0<r<R}\fint_{B(a,r)} |V(z, B(a,r))|\,dm(z)=0
\end{equation}
where we denote, for any disc $B\subset\Omega$,
$$V(z,B)=\sup_{w\neq 0}\frac{|\cH(z,w)-\fint_B\cH(\zeta,w)\,dm(\zeta)|}{|w|}.$$

In particular, this assumption is satisfied for fields $\cH$ satisfying the H\"older continuity condition \eqref{Hcondition}. We have chosen to formulate a more general result by considering the weaker $\mathrm{VMO}$-condition for the possibility of having broader applications in the future.

\begin{thm}\label{vmolocalestimate}
Suppose that $q\in (2,\infty)$ and that data $g\in L^q_{\mathrm{loc}}(\Omega, \C)$ is given. Let $x_0\in \Omega$ be fixed.  If $f\in W^{1,2}_{\mathrm{loc}}(\Omega, \C)$ is such that
$$f_{\zbar}= \cH(z,f_{z}) + g \qquad \text{a.e. $z\in \Omega$},$$
then there exist positive numbers $d_0$ and $C = C(K, q)$ such that the estimate
$$\left(\int_{B_r} |Df|^q\right)^\frac1q\leq \frac{C(K, q)}{r}\left(\int_{2 B_r}|f|^q\right)^\frac1q + C(K, q)\left(\int_{2 B_r}|g|^q\right)^\frac1q$$
holds for each disc $B_r=B(x_0,r)$ of radius $r< d_0$. Here $d_0$ depends on $x_0$, $K$, $q$ and  the $\mathrm{VMO}$ modulus of continuity of $\cH$.
\end{thm}

\begin{proof}
Throughout the proof we consider numbers $\lambda$, $k_0$, $\delta$ and $d_0$. The numbers $k_0 > 1$, $\delta > 0$ and $d_0 > 0$ are fixed but we will determine their values later. We will consider $\lambda$ as a free parameter in the interval $(1, 2)$.

Our first restriction is
\begin{equation}\label{incl}
\frac{\delta+1}{\delta}\,k_0\,\lambda\,d_0<d(x_0,\partial\Omega).
\end{equation}
Once $\delta$ and $k_0 $ are chosen the above restriction can always be satisfied by reducing $d_0$. Indeed,  since $\lambda <  2$, we choose $d_0<\frac{\delta}{\delta + 1}\,\frac{d(x_0,\partial\Omega)}{2k_0}$. Let us call $R_0=\frac{k_0\lambda d_0}{\delta}$. For every radius $0<R<R_0$ and every center $a\in B(x_0, \delta R)$, we will denote $B_{a, R}=B(a,R)$. Condition \eqref{incl} ensures that $B_{a, R}\subset \Omega$ and therefore the quantity $V(\cdot, B_{a, R})$ is well defined.

\smallskip

We first localize the problem. Let $r \in (0, d_0)$ and consider $\eta\in C^\infty_0(\Omega)$ be such that $\chi_{B_r}\leq \eta\leq \chi_{\lambda B_r}$ and $\|D\eta\|_\infty\leq \frac{C_1}{(\lambda-1)r}$, for some universal constant $C_1$, and set $h=\eta f$. Then
$$
h_{\zbar} = \cH(z, h_{z}) + G
$$ 
where
$$G = f \,\eta_{\zbar} + \eta\,\cH(z, f_{z}) + \eta g - \cH(z, h_{z}).$$

The second step consists of freezing the coefficients in $B_{a, R}$. This is done with the help of an autonomous equation, by setting 
$$\cH_{B_{a, R}}(w)=\fint_{B_{a, R}}\cH(z,w)\,dm(z), \qquad w\in\C.$$
This function $\cH_{B_{a, R}}$ defines an autonomous, $k$-elliptic Beltrami functional. Call $F$ the unique solution of the following autonomous Riemann-Hilbert problem
$$
\begin{cases}
F_{\zbar} = \cH_{B_{a, R}}(F_{z}),& z \in B_{a, R}\\ \Re\, F= \Re\, h, & z \in \partial B_{a, R}.\end{cases}
$$
Such solution $F$ exists and is unique and moreover it belongs to $C^{1,1/K}_{\mathrm{loc}}$, see Proposition \ref{splitting} and Corollary \ref{holder1K}. % Proposition 2.4 and Corollary 2.3 in \cite{ACFJK}. 
More precisely, for every $0<\delta<1$ we have the following Schauder estimate
\begin{equation}\label{schauder-eq}
\|F_{\zbar}-(F_{\zbar})_{B_{a, \delta R}}\|_{L^2(B_{a, \delta R})}\leq C_2(K)\,\delta^{1 + \frac1K}\,\|Dh\|_{L^2(B_{a, R})},
\end{equation}
for some constant $C_2(K)$ that only depends on $K$.

For every number $t\in(2,q)$ we have 
$$\aligned
\|&(h-F)_{\zbar}\|_{L^2(B_{a, R})} \\
&=\|\cH(z, h_{z}) + G - \cH_{B_{a, R}}(F_{z})\|_{L^2(B_{a, R})} \\
&\leq\|\cH(z, h_{z})-\cH_{B_{a, R}}(h_{z})\|_{L^2(B_{a, R})}\\
&\qquad+\|\cH_{B_{a, R}}(h_{z}) - \cH_{B_{a, R}}(F_{z})\|_{L^2(B_{a, R})}+\|G\|_{L^2(B_{a, R})} \\
&\leq\||h_{z}|\,V(\cdot, B_{a, R})\|_{L^2(B_{a, R})}+k\| h_{z}-F_{z}\|_{L^2(B_{a, R})} +\|G\|_{L^2(B_{a, R})} \\
&\leq\|h_{z}\|_{L^t(B_{a, R})}\,\|V(\cdot, B_{a, R})\|_{L^{\frac{2t}{t-2}}(B_{a, R})}\\
&\qquad +k\|  (h- F)_{\zbar}\|_{L^2(B_{a, R})}+\|G\|_{L^2(B_{a, R})} \\
\endaligned$$
whence
\begin{equation}\aligned
\|(h-F)_{\zbar}\|_{L^2(B_{a, R})} \leq &\ \frac{1}{1-k}\|h_{z}\|_{L^t(B_{a, R})}\,\|V(\cdot, B_{a, R})\|_{L^{\frac{2t}{t-2}}(B_{a, R})}\\&\ +\frac{1}{1-k}\,\|G\|_{L^2(B_{a, R})} .
\endaligned
\end{equation}
We use this estimate to control the local Fefferman-Stein maximal function of $h_{\zbar}$. Recall that this maximal function is defined by
$$\cM^\sharp_{2,R_0} (u) (x)=\sup_{0<R<R_0}\left(\fint_{B(x,R)}|u-u_{B(x,R)}|^2\right)^\frac12.$$ 
First, with the help of the Schauder estimate for $F$ \eqref{schauder-eq}, we have for any $0<\delta<1$ that
$$\aligned
\|h_{\zbar} -& (h_{\zbar})_{B_{a, \delta R}}\|_{L^2(B_{a, \delta R})}\\
&\leq 2 \|h_{\zbar} - F_{\zbar}\|_{L^2(B_{a, \delta R})}+\|F_{\zbar} - (F_{\zbar})_{B_{a, \delta R}}\|_{L^2(B_{a, \delta R})}\\
&\leq \frac{2}{1-k}\,\|h_{z}\|_{L^t(B_{a, R})}\,\|V(\cdot, B_{a, R})\|_{L^{\frac{2t}{t-2}}(B_{a, R})} \\
&\qquad +\frac{2}{1-k}\,\|G\|_{L^2(B_{a, R})}+C_2(K)\,\delta^{1 + \frac1K}\,\|Dh\|_{L^2(B_{a, R})}.
\endaligned$$
We now multiply by $\frac1{\delta R}$ and take supremum in $R$ over $(0,R_0)$. We obtain
$$
\aligned
\cM^\sharp_{2,\delta R_0}(h_{\zbar})(a)
&\leq  \frac{2}{\delta(1-k)}\,\cM_t(h_{z})(a)\,\sup_{0<R<R_0}\|V(\cdot, B_{a, R})\|_{L^{\frac{2t}{t-2}}(B_{a, R})}\\
&\qquad +\frac{2}{\delta(1-k)}\,\cM_2(G)(a) 
+C_2(K)\,\delta^\frac1K\,\cM_2(Dh)(a),
\endaligned
$$
where $\cM_t (u)(x)=\cM(u^t)(x)^\frac1t$ and $\cM$ is the Hardy-Littlewood maximal function. Now, we raise to the power $q$, integrate with respect to $a$ over the set $B(x_0, \delta R_0)$ and take power $\frac1q$. We obtain
$$
\aligned
&\|\cM^\sharp_{2,\delta R_0}(h_{\zbar})\|_{L^q(B(x_0, \delta R_0))}\\
&\leq  \frac{2}{\delta(1-k)}\,\|\cM_t(h_{z})\|_{L^q(B(x_0, \delta R_0))}\!\!\sup_{a\in B(x_0, \delta R_0)}\sup_{0<R<R_0}\|V(\cdot, B_{a, R})\|_{L^{\frac{2t}{t-2}}(B_{a, R})}\\
&\qquad+\frac{2}{\delta(1-k)}\,\|\cM_2(G)\|_{L^q(B(x_0, \delta R_0))} 
+C_2(K)\,\delta^\frac1K\,\|\cM_2(Dh)\|_{L^q(B(x_0, \delta R_0))}.
\endaligned
$$
Since $\delta R_0= k_0\lambda d_0>\lambda d_0$ and $\supp(h)\subset \lambda B_r \subset \lambda B_{d_0}$, we can use at the right hand side the boundedness of $\cM$ on the whole plane, 
\begin{equation}\label{maximals}
\aligned
&\|\cM^\sharp_{2,\delta R_0}(h_{\zbar})\|_{L^q(B(x_0, \delta R_0))}\\
&\leq  \frac{2}{\delta(1-k)}\,\|\cM_t\|_{L^q(\C)}\,\|h_{z}\|_{L^q(\C)}\!\!\sup_{a\in B(x_0, \delta R_0)}\sup_{0<R<R_0}\|V(\cdot, B_{a, R})\|_{L^{\frac{2t}{t-2}}(B_{a, R})}\\
&\qquad+\frac{2}{\delta(1-k)}\,\|\cM_2\|_{L^q(\C)}\,\|G\|_{L^q(\C)} 
+C_2(K)\,\delta^\frac1K\,\|\cM_2\|_{L^q(\C)}\,\|Dh\|_{L^q(\C)}.
\endaligned
\end{equation}

We now look at the last term on the right hand side of \eqref{maximals}. The boundedness of the Beurling transform $\mathcal{S}$ on $L^q(\C)$ allows us to write
$$\aligned
C_2(K)\,\delta^\frac1K\,\|\cM_2\|_{L^q(\C)}\,&\|Dh\|_{L^q(\C)}\\&\leq C_2(K)\,(1+\|\mathcal{S}\|_{L^q(\C)})\,\delta^\frac1K\,\|\cM_2\|_{L^q(\C)}\,\|h_{\zbar}\|_{L^q(\C)}.\endaligned$$
At this point we fix $\delta>0$ small enough so that we attain
\begin{equation}\label{choosedelta}
C_2(K)\,\delta^\frac1K\,\|\cM_2\|_{L^q(\C)}\,\|Dh\|_{L^q(\C)}\leq \epsilon\|h_{\zbar}\|_{L^q(\C)}
\end{equation}
for $\epsilon = \epsilon(q) > 0$ to be chosen later.
Remarkably, the chosen value of $\delta$ depends on $K$ and $q$ and nothing else. 

Now, the term on the left hand side of \eqref{maximals} can also be bounded from below with the help of the following inequality for the local Fefferman-Stein maximal function
$$
\|\cM^\sharp_{2,\delta R_0}(h_{\zbar})\|_{L^q(B(x_0, \delta R_0))}\geq C_3(q)\,\| h_{\zbar} \|_{L^q(B(x_0,\frac{ \delta R_0}{k_0}))}
$$
which holds for a number $k_0\geq 2$ (a proof of this fact follows as in  \cite[Lemma 2.4]{Kinnunen-Zhou}). Thus $k_0$ is fixed at this point. We now choose $\epsilon = \frac{1}{4}C_3(q)$. 

Recall that $\lambda B_r \subset B(x_0, \frac{ \delta R_0}{k_0})$, since $\delta R_0= k_0\lambda d_0$. Now, using \eqref{choosedelta},  equation \eqref{maximals} reduces to
\begin{equation}\label{maximals2}
\aligned
&\frac34 C_3(q)\,\|h_{\zbar}\|_{L^q(\C)}\\
&\leq  \frac{2}{\delta(1-k)}\,\|\cM_t\|_{L^q(\C)}\,\|h_{z}\|_{L^q(\C)}\,\sup_{a\in B(x_0, \delta R_0)}\sup_{0<R<R_0}\|V(\cdot, B_{a, R})\|_{L^{\frac{2t}{t-2}}(B_{a, R})}\\
&\qquad+\frac{2}{\delta(1-k)}\,\|\cM_2\|_{L^q(\C)}\,\|G\|_{L^q(\C)}. 
\endaligned
\end{equation}

We now look at the term with the supremum. Our assumption \eqref{incl} tells us that $B(a,\delta R_0)$ is always a subset of the compact set
$$\mathcal{K}=\overline{B\left(x_0, \frac{\delta \,d(x_0, \partial\Omega)}{\delta+1}\right)}.$$
Also the $k$-Lipschitz property of $\cH$ says that $V(\cdot, B_{a, R})\leq k$. Hence
$$\aligned
\sup_{a\in B(x_0, \delta R_0)}\sup_{0<R<R_0}&\|V(\cdot, B_{a, R})\|_{L^{\frac{2t}{t-2}}(B_{a, R})}\\
&\leq C_4(K,t)\,\sup_{a\in \mathcal{K}}\sup_{0<R<R_0}\|V(\cdot, B_{a, R})\|_{L^1(B_{a, R})}^\frac{t-2}{2t}\\
&\leq \frac14\,\frac{C_3(q)}{\frac{2}{\delta(1-k)}\,\|\cM_t\|_{L^q(\C)}}.
\endaligned$$
Above the last inequality follows by our $\mathrm{VMO}$ assumption \eqref{locunifvmo} on $\cH$. Indeed, since $R_0=\frac{k_0\lambda d_0}{\delta}$, we are free to choose $d_0$ small enough so that the last inequality above holds. Especially, the choice of $d_0$ depends on $\mathcal{K}$, $K$, $t$, $q$ and $\mathrm{VMO}$ property of $\cH$. As a consequence, \eqref{maximals2} gets converted into
\begin{equation}\label{maximals3}
\aligned
\frac12 C_3(q)\,\|h_{\zbar}\|_{L^q(\C)}
&\leq\frac{2}{\delta(1-k)}\,\|\cM_2\|_{L^q(\C)}\,\|G\|_{L^q(\C)} 
\endaligned
\end{equation}
or simply
$$
\|h_{\zbar}\|_{L^q(\C)}\leq C_5(K,q) \,\|G\|_{L^q(\C)}. 
$$
The boundedness of the Beurling transform on $L^q(\C)$ and the fact that $h\in W^{1,q}(\C)$ allows us to conclude that
$$
\|D h\|_{L^q(\C)}\leq C_6(K,q) \,\|G\|_{L^q(\C)}. 
$$
On the other hand, we can use the definition of $G$ and the fact that $\supp(\eta)\subset \lambda B_r$ to obtain that
$$\aligned
G
&= f \,\eta_{\zbar} + \eta g + \cH(z, \eta \,f_z)-\cH(z, \eta\, f_{z}+ f\,\eta_{z})\\ & \qquad +\eta \cH(z,f_{z})  - \cH(z, \eta \,f_{z})\\
&= f \,\eta_{\zbar} + \eta g + \cH(z, \eta \,f_{z})-\cH(z, \eta\, f_{z}+ f\,\eta_{z})\\ &\qquad +\chi_{\lambda B_r}(\eta \cH(z,f_{z})  - \cH(z, \eta \,f_{z}))\\
&= f \,\eta_{\zbar} + \eta g + \cH(z, \eta \,f_{z})-\cH(z, \eta\, f_{z}+ f\,\eta_{z})\\
&\qquad+\chi_{\lambda B_r}(\eta-1) \cH(z,f_{z}) +\chi_{\lambda B_r}(\cH(z,f_{z}) - \cH(z, \eta \,f_{z}))
\endaligned$$
and thus
$$\aligned
|G|&\leq  |f \,\eta_{\zbar}| + |\eta g| + k\,|f\,\eta_{z}|+2k\,|\eta-1|\,|f_{z}|\,\chi_{\lambda B_r}.
\endaligned$$
After recalling that $\eta=1$ on $\chi_{B_r}$ we are left with
$$\aligned
\|G\|_{L^q(\C)}\leq&\, \|f\, \eta_{\zbar}\|_{L^q(\lambda B_r)}+\|g \eta\|_{L^q(\lambda B_r)}\\&\, +k \|f \,\eta_{z}\|_{L^q(\lambda B_r)}+ 2k\| f_{z}\|_{L^q(\lambda B_r\setminus B_r)}.\endaligned
$$
Therefore
$$\aligned
\|D h\|_{L^q(\C)}
&\leq C_6(K,q)\left(\|f\, D\eta\|_{L^q(\lambda B_r)}+\|g \eta\|_{L^q(\lambda B_r)}+ 2k\|f_{z}\|_{L^q(\lambda B_r\setminus B_r)}\right)\\
&\leq C_6(K,q)\,\|f\, D\eta\|_{L^q(\lambda B_r)}+C_6(K,q)\,\|g \eta\|_{L^q(\lambda B_r)}\\& \quad \, + C_7(K,q)\,\|D f\|_{L^q(\lambda B_r\setminus B_r)}. 
\endaligned$$
Since $h=\eta f$, this implies that
$$\aligned
\|D f\|_{L^q(B_r)}
&\leq \|\eta D f\|_{L^q(\C)}\\
&\leq \|Dh\|_{L^q(\C)}+\|f \, D\eta \|_{L^q(\C)}\\
&\leq (C_6 +1) \,\|f\, D\eta\|_{L^q(\lambda B_r)}+C_6\,\|g \eta\|_{L^q(\lambda B_r)}\\&\quad \, +C_7\,\|D f\|_{L^q(\lambda B_r\setminus B_r)}.
\endaligned$$
It just remains to fill the hole, that is, to add the term  $C_7\,\|D f\|_{L^q(B_r)}$ at both sides to obtain
$$\aligned
(C_7+1)\|D f\|_{L^q(B_r)}\leq\, & (C_6+1) \,\|f\, D\eta\|_{L^q(\lambda B_r)}+C_6\,\|g \eta\|_{L^q(\lambda B_r)}\\&+ C_7\,\|D f\|_{L^q(\lambda B_r)} 
\endaligned $$
that is
$$\aligned
\|D f\|_{L^q(B_r)}
&\leq \tau\,\|D f\|_{L^q(\lambda B_r)} +  C_8(K, q)\,\|f\, D\eta\|_{L^q(\lambda B_r)}\\& \quad \, +C_9(K, q)\,\|g \eta\|_{L^q(\lambda B_r)} \\
&\leq \tau\,\|D f\|_{L^q(\lambda B_r)} +  \frac{C_{10}(K, q)}{(\lambda-1)r}\,\|f\|_{L^q(\lambda B_r)}+C_9\,\|g \|_{L^q(\lambda B_r)}
\endaligned$$
with $\tau=\frac{C_7}{C_7+1} < 1$. The restrictions on $\lambda$ are such that 
$$r<\lambda r< 2 r.$$
 Thus, a classical iteration argument gives (see \cite[Lemma 6.1, p. 191]{Giusti})
$$\aligned
\|D f\|_{L^q(B_r)}
&\leq  \frac{C_{11}(K, q)}{r}\,\|f\|_{L^q(2 B_r)}+C_{11}(K, q)\,\|g \|_{L^q(2 B_r)}
\endaligned$$
as desired.
\end{proof}


\begin{thebibliography}{1}




\bibitem{ACFJ}
K.~Astala, A.~Clop, D.~Faraco, and J.~J{\"a}{\"a}skel{\"a}inen.
\newblock Manifolds of Quasiconformal Mappings and the Nonlinear Beltrami Equation. 
\newblock {\em J. Anal. Math.}, in press.
\newblock arXiv:1412.4046.

\bibitem{ACFJK} 
K.~Astala, A.~Clop, D.~Faraco,  J.~J{\"a}{\"a}skel{\"a}inen, and A.~Koski.
\newblock Nonlinear Beltrami operators, Schauder estimates and bounds for the Jacobian.
\newblock {\em Ann. I. H. Poincar\'e -- AN} 34(6) (2017), 1543--1559.



\bibitem{Uniqueness}
K.~Astala, A.~Clop, D.~Faraco, J.~J{\"a}{\"a}skel{\"a}inen, and
  L.~Sz{\'e}kelyhidi, Jr.
\newblock Uniqueness of normalized homeomorphic solutions to nonlinear
  {B}eltrami equations.
\newblock {\em Int. Math. Res. Not. IMRN} 2012(18) (2012), 4101--4119.


\bibitem{convexintegration}
K.~Astala, D.~Faraco, and L.~Sz\'ekelyhidi Jr.
\newblock Convex integration and the $L^p$ theory of elliptic equations. 
\newblock {\em Ann. Sc. Norm. Super. Pisa Cl. Sci. (5)} 7 (2008), 1--50. 


\bibitem{AIM}
K.~Astala, T.~Iwaniec, and G.~Martin.
\newblock {\em Elliptic {P}artial {D}ifferential {E}quations and
  {Q}uasiconformal {M}appings in the {P}lane}, Princeton
  Mathematical Series 48.
\newblock Princeton University Press, Princeton, NJ, 2009.

\bibitem{ADPZ}
K.~Astala, E.~Duse, I.~Prause, and X.~Zhong.
\newblock Dimer models and conformal structures.
\newblock Preprint.



\bibitem{Baernstein-Kovalev}
A. Baernstein II and L. V. Kovalev.
\newblock On H\"older regularity for elliptic equations of non-divergence type in the plane.
\newblock {\em Ann. Sc. Norm. Super. Pisa Cl. Sci. (5)} 4 (2005), 295--317. 



\bibitem{Bers-Nirenberg}
L. Bers and L. Nirenberg.
\newblock On linear and nonlinear elliptic boundary value problems in the plane.
\newblock In {\em Conv. Int.le ``Eq. Lineari a Derivate Parziali'' (Trieste 1954), Roma},
1955, pp. 141--167.

\bibitem{Bojarski}
B. Bojarski.
\newblock  Generalized solution of a system of first order differential equations of elliptic type with discontinuous coefficients.
\newblock {\em Math. Sb.} 43 (1957), 451--503.
\newblock English translation in Rep. Univ. Jyv\"askyl\"a Dept. Math. Stat. 118, 2009.




\bibitem{Caccioppoli}
R. Caccioppoli.
\newblock Fondamenti per una teoria generale delle funzioni pseudoanalitiche di una
variabile complessa.
\newblock {\em  Rend, Acc. Naz. Lincei} 13 (1952), 197--204.

\bibitem{Faraco}
 D. Faraco. 
 \newblock Tartar conjecture and Beltrami operators. 
 \newblock {\em Michigan Math. J.} 52(1) (2004), 83--104.

\bibitem{Faraco-Kristensen}
 D. Faraco and J. Kristensen. 
 \newblock Compactness versus regularity in the calculus of variations. 
 \newblock {\em Discrete Contin. Dyn. Syst. Ser. B} 17(2) (2012), 473--485.
 
 \bibitem{tartar}
D.~Faraco and L.~Sz\'ekelyhidi Jr.
\newblock Tartar's conjecture and localization of the quasiconvex hull in {$\R^{2\times 2}$}.
\newblock {\em Acta Math.}, 200 (2008), 279--305.

\bibitem{Finn-Serrin}
R. Finn and J. Serrin.
\newblock On the H\"older continuity of quasiconformal and elliptic mappings.
\newblock {\em Trans. Amer. Math. Soc.} 89 (1958), 1--15.




\bibitem {Gia} 
  M.~Giaquinta.
\newblock {\em Multiple integrals in the calculus of variations and nonlinear
              elliptic systems}, Annals of Mathematics Studies 105.
\newblock Princeton University Press, Princeton, NJ., 1983.

  \bibitem {Giusti} 
  E.~Giusti.
\newblock {\em Direct methods in the calculus of variations}.
\newblock World Scientific Publishing Co., Inc., River Edge, NJ, 2003.
  


\bibitem{Hinkkanen-Martin}
A.~Hinkkanen and G.~Martin.
\newblock Quasiregular families bounded in $L^p$ and elliptic estimates.
\newblock arXiv:1806.00758.

\bibitem {Iwaniec-Sbordone} 
T.~Iwaniec and C.~Sbordone.
\newblock Quasiharmonic fields.
\newblock {\em Ann. I. H. Poincar\'e -- AN} 18(5) (2001), 519--572.

\bibitem {Kinnunen-Zhou}
J.~Kinnunen and S.~Zhou.
\newblock  A Local estimate for nonlinear equations with discontinuous
              coefficients.
\newblock {\em Comm. Partial Differential Equations} 24(11-12) (1999), 2043--2068.

\bibitem{Koskela} 
P. Koskela.
\newblock The Degree of Regularity of a Quasiconformal Mapping. 
\newblock {\em Proc. Am. Math. Soc.} 122(3) (1994), 769--772.

\bibitem{Kovalev} 
L. V. Kovalev.
\newblock Quasiconformal geometry of monotone mappings.
\newblock {\em J. Lond. Math. Soc.} 75 (2007), 391--408.

\bibitem{Leonetti-Nesi}
F. Leonetti and  V. Nesi.
\newblock  Quasiconformal solutions to certain first order systems and the proof of a conjecture of G. W. Milton.
\newblock{ \em  J. Math. Pures Appl.}  76(2) (1997), 109--124.


\bibitem{LU} 
O.~A.~Ladyzhenskaya and N.~N.~Ural'tseva.
\newblock {\em Linear and Quasilinear Elliptic Equations}.
\newblock Academic Press, New York, 1968.

\bibitem{L-L} 
J. Leray and J.-L. Lions.
\newblock Quelques r\'esultats de Vi\u{s}ik sur les probl\`emes
elliptiques non lin\'eaires par les m\'ethodes
de Minty-Browder
\newblock {\em Bulletin de la S. M. F.} 93 (1965), 97--107.

\bibitem{Martin}
\newblock G. Martin.
\newblock Super regularity for Beltrami systems.
\newblock arXiv:1903.01008.


\bibitem{Morrey}
\newblock C.~B.~Morrey, Jr.,
\newblock On the solutions of quasi-linear elliptic partial differential
              equations,
\newblock {\em Trans. Amer. Math. Soc.} 43(1) (1938), 126--166.





\bibitem{s1} 
J.~Schauder.
\newblock \"{U}ber lineare elliptische {D}ifferentialgleichungen zweiter
              {O}rdnung. 
\newblock {\em Math. Z.} 38(1) (1934), 257--282.


\bibitem{s2} 
J.~Schauder.
\newblock Numerische Absch\"atzungen in elliptischen linearen Differentialgleichungen.
\newblock {\em Studia Math.} 5 (1935), 34--42.

\bibitem{Sverak} 
V. \v{S}ver\'ak. 
\newblock On Tartar's conjecture. 
\newblock {\em Ann. I. H. Poincar\'e -- AN} 10(4) (1993), 405--412.

\bibitem{Vekua} 
I. N. Vekua. 
\newblock {\em Generalized Analytic Functions}, Pure and Applied Mathematics 25.
\newblock  Pergamon Press, Oxford, 1959

\end{thebibliography}
\end{document}